%% file: bending_kleinian.tex
\documentclass[notitlepage,oneside,a4paper]{amsart}
\usepackage{mathrsfs}
\usepackage{amsmath,amssymb,enumerate}
\usepackage{epsfig,fancyhdr,color}
\usepackage{epstopdf}
\usepackage{amssymb}
\usepackage{amsmath,amsthm}
\usepackage{mathtools}
\usepackage{latexsym}
\usepackage{amscd}
\usepackage{psfrag}
\usepackage{graphicx}
\usepackage{epsf}
\usepackage[utf8]{inputenc}
\usepackage[all]{xy}
\usepackage{tikz}
\usepackage{tikz-cd}
\usepackage{prettyref}
\usepackage{subfigure}
\usepackage{float}
\usepackage{color}
\usepackage{array}
\usepackage{cite}
\usepackage[hidelinks]{hyperref}
\usepackage{enumitem}
\usepackage{faktor}
\usepackage{booktabs}

\input{commands.tex}
\newtheorem{thmintro}{Theorem}

\newtheorem{corintro}[thmintro]{Corollary}
\usepackage[totalwidth=16cm,totalheight=24cm]{geometry}

\title[Bending parameterization]{Bending parameterization of geometrically finite hyperbolic manifolds}

\author{Bruno Dular}
\address{Bruno Dular:
University of Luxembourg, FSTM, Department of Mathematics, 
Maison du nombre, 6 avenue de la Fonte,
L-4364 Esch-sur-Alzette, Luxembourg}
\email{bruno.dular@uni.lu}

\date{September 24, 2026}

\begin{document}
\begin{abstract}
	We show that non-Fuchsian geometrically finite hyperbolic structures on a given hyperbolizable 3-manifold $M$ are uniquely determined, up to isotopy, by their bending laminations. Consequently, the bending map from the space of non-Fuchsian geometrically finite structures on $M$, endowed with the strong topology, to the space of bending laminations, endowed with Lecuire's tubular topology, is a homeomorphism.

	This extends the convex co-compact case established with Schlenker. The proof combines hyperbolic Dehn filling, continuity and properness of the bending map (Lecuire) and real-analyticity of its fibres (Bonahon). We also establish a criterion for contractibility of the boundary fibres of a continuous extension of a homeomorphism, extending Finney's theorem to a boundary setting. This topological result may be of independent interest.
\end{abstract}

\maketitle

\setcounter{tocdepth}{1}
\tableofcontents

\section{Introduction}\label{sec:introduction}

Let $M$ be an orientable 3-manifold that admits a complete hyperbolic metric and has finitely generated $\pi_1(M)$. By the tameness theorem~\cite{agol2004tamenesshyperbolic3manifolds,MR2188131}, $M$ is homeomorphic to the interior of a compact 3-manifold with boundary $\overline{M}$. Let $\partial M$ denote the union of the non-toroidal components of $\partial\overline{M}$, which we assume to be nonempty. A geometrically finite hyperbolic structure $\sigma$ on $M$ carries two kinds of boundary data on $\partial M$: the conformal structure at infinity $\mm(\sigma)$ and the bending lamination $\bb(\sigma)$ of the boundary of its convex core. The convex core boundary is a hyperbolic surface bent along a geodesic lamination, and $\bb(\sigma)$ records the amount of bending as a transverse measure~\cite{MR0903850,MR0903852}. Rank one cusps are accounted for by closed leaves of weight $\pi$, so that $\bb(\sigma)$ is an element of the space $\cML(\partial M)$ of measured laminations on $\partial M$.

The conformal structure is well understood: by the Ahlfors--Bers theory, $\sigma$ is determined, within its quasi-conformal deformation space, by $\mm(\sigma)\in\cT(\partial_{\infty}\sigma)$. Since $\cML(\partial M)$ and $\cT(\partial M)$ have the same dimension, it is natural to ask whether $\bb(\sigma)$ determines $\sigma$ as well. This question is attributed to Thurston. We answer this affirmatively for non-Fuchsian geometrically finite structures.

\subsection*{Previous results}
Fuchsian structures, when $M$ admits some, are not distinguished by their bending laminations (Remark~\ref{rem:bending_fuchsian}) and are excluded throughout: we write $\cGF^{\nf}(M)$ for the space of non-Fuchsian geometrically finite structures on $M$, up to isotopy. It is endowed with the topology of \emph{strong convergence}.

The image of $\bb\colon\cGF^{\nf}(M)\to\cML(\partial M)$ was determined by Bonahon--Otal~\cite{MR2144972} and Lecuire~\cite{MR2207784} (Theorem~\ref{thm:realizability_criterion}): a measured lamination $\lambda$ is \emph{realizable} (i.e.\ in the image of $\bb$) if and only if it has no closed leaf of weight $>\pi$ and satisfies an annulus condition and a disk condition. We denote by $\cP(M)$ the set of such laminations, and by $\cP_{\nc}(M)\subset\cP(M)$ those without closed leaf of weight $\geq\pi$. Bonahon--Otal also proved that a rational lamination $\lambda\in\cP(M)$ is realized by a unique hyperbolic structure, using rigidity results for cone manifolds~\cite{MR1622600}.

Injectivity of $\bb$ in general does not follow from that strategy. Keen--Series~\cite{MR2052972} and Series~\cite{MR2258745} established it for punctured torus groups, and Bonahon~\cite{MR2186972} near the Fuchsian locus of quasi-Fuchsian space. Lecuire~\cite{MR2464096,lecuire2025propernessbendingmap} showed that $\bb$ becomes continuous and proper once $\cP(M)$ is endowed with a suitable \emph{tubular topology}, described below. Recently, with Schlenker, we proved that $\bb$ is injective on convex co-compact structures~\cite{dularschlenker2024pleating}. Together with the results above, it follows that $\bb\colon\cCC^{\nf}(M)\to\cP_{\nc}(M)$ is a homeomorphism when $M$ admits convex co-compact structures (Theorem~\ref{thm:injectivity_convex_cocompact}).

\subsection*{Main results}
The present paper concerns geometrically finite hyperbolic manifolds and removes the convex co-compactness assumption. We proceed in two steps. A geometrically finite structure is \emph{minimally parabolic} if it has no rank-one cusps (when $\pi_1(M)$ contains no $\Z^2$, this means convex co-compact), and the space of minimally parabolic structures on $M$ is denoted $\cMP(M)$. The first step allows rank-two cusps.

\begin{thmintro}[Theorem~\ref{thm:minimally_parabolic}]\label{thmA}
	Let $M$ be a hyperbolizable 3-manifold. The bending map
	\begin{equation*}
		\bb\colon\cMP^{\nf}(M)\longrightarrow\cML(\partial M)
	\end{equation*}
	is a homeomorphism onto its image $\cP_{\nc}(M)$.
\end{thmintro}

When $M$ has rank-two cusps, it admits no Fuchsian structure. When it has none, it is the convex co-compact case of~\cite{dularschlenker2024pleating}. By the density theorem~\cite{MR2079598,MR2821565,MR3001608}, $\cMP(M)$ is dense in the space of all hyperbolic structures on $M$.

The second step allows rank-one cusps. Here the target must carry Lecuire's tubular topology~\cite{lecuire2025propernessbendingmap}, which we recall. A measured lamination $\lambda\in\cP(M)$ with closed leaves of weight $\pi$ is the bending lamination of a structure whose parabolic locus $P$ is the union of these leaves. A measured lamination $\mu$ is close to $\lambda$ in the tubular topology if it is close to $\lambda$ in the usual sense away from $P$, and if near each component $c$ of $P$ it is ``nearly parallel'' to $c$ and carries total transverse measure at least $\pi-\epsilon$ across $c$, with no upper bound (Definition~\ref{def:tubular_topology}). Indeed, a sequence of structures pinching $c$ to a cusp may have bending laminations spiraling more and more around $c$: a strongly bent torus around $c$ then escapes at infinity and disappears in the geometric limit~\cite{MR2144972,MR2464096,lecuire2025propernessbendingmap}.\footnote{This is analogous to the fact that, in the Fenchel--Nielsen coordinates of the conformal boundary, such a \emph{strongly convergent} sequence of structures pinching $c$ may have unbounded twist parameters about $c$, as long as the length of $c$ tends to $0$.} Let $\cBL(M)$ denote $\cP(M)$ with this topology.

\begin{thmintro}[Theorem~\ref{thm:injectivity_bending_map_rank_one_cusps}]\label{thmB}
	Let $M$ be a hyperbolizable 3-manifold. The bending map
	\begin{equation*}
		\bb\colon\cGF^{\nf}(M)\longrightarrow\cBL(M)
	\end{equation*}
	is a homeomorphism.
\end{thmintro}

Continuity, properness and surjectivity were known by the aforementioned works. The new content is injectivity. In particular, quasi-Fuchsian structures on $S\times\R$, for $S$ a hyperbolizable surface with punctures, form a stratum of the space of geometrically finite structures on a handlebody, and we obtain the following extension of the punctured torus case~\cite{MR2258745}, see Section~\ref{sec:punctured_surface_groups}.

\begin{corintro}\label{corC}
	Let $S$ be a hyperbolizable surface of finite type. The bending map
	$$\bb\colon\cQF^{\nf}(S)\to\cML(S)\times\cML(S)$$
	is a homeomorphism onto its image.
\end{corintro}

\subsection*{Contractibility of boundary fibres}
The passage from minimally parabolic to geometrically finite structures leads to a general topological question: what can be said about the fibres of a continuous extension of a homeomorphism to a boundary? We prove that, under suitable assumptions, the nonempty boundary fibres are contractible.

This extends Finney's theorem on limits of injective maps to a boundary setting~\cite{MR0224087}. We state it here in abbreviated form; see Section~\ref{sec:boundary_of_homeomorphisms} for the precise meaning of the hypotheses.
\begin{thmintro}[Theorem~\ref{thm:fibres_contractible_boundary_of_homeo}]\label{thmD}
	Let $\Phi\colon A\to B$ be a continuous map between first-countable Hausdorff spaces, and let $A_{\infty}\subset A$ be a closed subset with $A_{\circ}\coloneqq A-A_{\infty}$ dense in $A$. Suppose that $\Phi\restr{A_{\circ}}$ is an open embedding, that $\Phi$ is proper onto its image $\Phi(A)$, that nonempty fibres of $\Phi\restr{A_{\infty}}$ are retracts of relatively compact neighborhoods in $A_{\infty}$. Assume moreover that:
	\begin{itemize}
		\item $A$ is ``locally a product'' near $A_{\infty}$;
		\item $\Phi$ has ``contractible image at infinity'': each $y\in\Phi(A_{\infty})$ has a basis of neighborhoods $V$ such that $V\cap\Phi(A_{\circ})$ is contractible.
	\end{itemize}	
	Then the nonempty fibres of $\Phi\restr{A_\infty}$ are contractible.
\end{thmintro}
Applied to the bending map, Theorem~\ref{thmD} gives contractibility of the fibres on each stratum with fixed parabolic locus. Since these fibres are compact real-analytic spaces, contractibility forces them to be singletons, yielding the injectivity in Theorem~\ref{thmB}.

\subsection*{Strategy}
The convex co-compact case, obtained with Schlenker in~\cite{dularschlenker2024pleating}, rests on two facts. First, fibres of $\bb$ are compact real-analytic spaces, by Bonahon's shear-bend coordinates~\cite{MR1413855} and Lecuire's properness~\cite{lecuire2025propernessbendingmap}. A contractible compact real-analytic space is a point, by the existence of a fundamental class~\cite{MR149503,MR278333}. Second, a theorem of Finney~\cite{MR0224087} shows that if a sequence of injective continuous maps between manifolds of the same dimension converges continuously, the compact fibres of the limit are contractible, as long as they are topologically ``nice'', e.g.\ \emph{absolute neighborhood retracts}. Injectivity thus follows from exhibiting $\bb$ as a continuous limit of injective maps, an ``approximation sequence''. In the convex co-compact case, this was done via approximating the convex core boundary and its bending lamination by smooth surfaces foliating the end and their third fundamental forms, respectively, using results of Labourie~\cite{MR1125669,MR1163450} and Schlenker~\cite{MR2208419}. However, these results are not known in the geometrically finite case, hence one must take a different route.

For Theorem~\ref{thmA}, the approximation sequence is given by hyperbolic Dehn filling. Filling the rank-two cusps of $M$ along slopes $\mathbf{s}^n\to\infty$ produces manifolds $M_n$ without rank-two cusps, and sending $\sigma\in\cMP(M)$ to the convex co-compact structure on $M_n$ with the same conformal boundary defines homeomorphisms $F_n\colon\cMP(M)\to\cCC(M_n)$. The maps $\bb\circ F_n$ are injective by~\cite{dularschlenker2024pleating} (away from the Fuchsian locus), and they converge continuously to $\bb$ (Proposition~\ref{prop:continuous_convergence_drilling}) when the slopes tend to infinity. The ingredients are the Brock--Bromberg drilling theorem~\cite{MR2079598} (see also~\cite{comar1996hyperbolic,MR4466646,MR4468992}), which implies strong convergence $F_n(\sigma_n)\to\sigma$ whenever $\sigma_n\to\sigma$, and an adaptation of Lecuire's proof of continuity of the bending map~\cite{MR2464096} to sequences of pleated surfaces in varying manifolds.

For Theorem~\ref{thmB}, one cannot use the same strategy as there is no ``canonical'' approximation sequence to create rank-one cusps. However, the structures with a given parabolic locus $P$ lie in a stratum of the boundary of minimally parabolic ones, on which $\bb$ is already known to be a homeomorphism (by \cite{dularschlenker2024pleating} or by Theorem~\ref{thmA}). The main ingredient is then Theorem~\ref{thmD}. Applied to the bending map on $\cMP^{\nf}(M)\cup\cGF^{\nf}(M;P)$, it yields contractibility of the fibres over the stratum with parabolic locus $P$, and real-analyticity concludes as before. The two main hypotheses to check are the following:
\begin{enumerate}
	\item A \emph{local product structure} on the domain (\S~\ref{sec:conformal_parameterization_of_geometrically_finite_structures}): it follows from the conformal parameterization of $\cGF(M)$, endowed with the strong topology, by the union of the strata of the augmented Teichm\"uller space of $\partial M$ indexed by pinchable multicurves (Theorem~\ref{thm:conformal_boundary_map_homeo}), combined with extended Fenchel--Nielsen coordinates. This parameterization is known in the incompressible boundary case~\cite{MR3134412}. For the compressible boundary case, we include a proof that combines the arguments of Anderson--Lecuire~\cite{MR3134412} and Lecuire~\cite{lecuire2025propernessbendingmap}. 
	\item A \emph{contractibility at infinity} condition (\S~\ref{sec:contractibility_at_infinity_geometrically_finite_structures}): this follows using Dehn--Thurston coordinates, the shape of tubular neighborhoods (Proposition~\ref{prop:contractibility_at_infinity_geometrically_finite}) and results of Lecuire~\cite{lecuire2025propernessbendingmap}.
\end{enumerate}

\subsection*{Geometrically infinite structures}
For structures with geometrically infinite ends, the bending lamination of the convex core boundary only sees the geometrically finite ends. The natural parameters combine bending laminations of the geometrically finite ends with ending laminations of the geometrically infinite ones. Realizability of such data for Kleinian surface groups was established by Baba--Ohshika~\cite{MR4651897}. Injectivity in the case of Kleinian surface groups with exactly one geometrically infinite end is treated in~\cite{dular2025bending}. The strategy of Theorem~\ref{thmB}, based on Theorem~\ref{thmD}, suggests an approach to the general case. Work toward such an extension is ongoing.

\subsection*{Organization}
Section~\ref{sec:background} collects background on the various deformation spaces of hyperbolic structures, on pleated surfaces and on bending laminations. Section~\ref{sec:boundary_of_homeomorphisms} proves the topological Theorem~\ref{thmD}. Section~\ref{sec:minimally_parabolic_structures} proves Theorem~\ref{thmA} and Section~\ref{sec:geometrically_finite_structures} proves Theorem~\ref{thmB}, ending with the case of punctured surface groups (Corollary~\ref{corC}).

\subsection*{Acknowledgements}
I wish to thank Kenneth Bromberg, Tommaso Cremaschi, Cyril Lecuire, Roman Prosanov, Jean-Marc Schlenker and Gabriele Viaggi for helpful discussions and comments. I was supported by FNR AFR grant 18890152 during the preparation of this paper.

\subsection*{AI use}
I used AI models to help draft and structure parts of the introduction and background. I also used them to proofread the manuscript and suggest improvements to its clarity. The original mathematical ideas, results, and proofs are my own.

\section{Background}\label{sec:background}
\subsection{Hyperbolic manifolds and Kleinian groups}\label{sec:hyperbolic_manifolds_and_klein_groups}
Let $M=\Int\overline{M}$, where $\overline{M}$ is a compact, oriented, hyperbolizable 3-manifold. As in the introduction, $\partial M$ denotes the union of the non-toroidal boundary components, which we assume to be nonempty, and $T$ denotes the union of the toroidal components.

A complete hyperbolic metric $\sigma$ on $M$ has a discrete faithful holonomy representation $\rho_{\sigma}\colon\pi_1(M)\to\PSL(2,\C)$, defined up to conjugacy. Write $\Gamma_{\sigma}=\rho_{\sigma}(\pi_1(M))$ and $M_{\sigma}=\HH^3/\Gamma_{\sigma}$, and denote its limit set and domain of discontinuity in $\CP^1$ by $\Lambda(\sigma)$ and $\Omega(\sigma)$, respectively. The convex core and conformal boundary are
\[
    C(\sigma)=\CH(\Lambda(\sigma))/\Gamma_{\sigma}, \qquad \partial_{\infty}\sigma=\Omega(\sigma)/\Gamma_{\sigma},
\]
where $\CH(\Lambda(\sigma))$ denotes the convex hull of $\Lambda(\sigma)$ inside $\HH^3$. See~\cite{MR1638795,MR3586015}.

We call $\sigma$ \emph{Fuchsian} if $C(\sigma)$ is a totally geodesic surface with possibly empty geodesic boundary. Otherwise, $\partial C(\sigma)$ is a locally convex pleated surface~\cite{MR0903850,MR0903852}, naturally identified up to isotopy with $\partial_{\infty}\sigma$ through the nearest-point retraction. Its components face the geometrically finite ends of $M_{\sigma}$.

The structure $\sigma$ is \emph{convex co-compact} if $C(\sigma)$ is compact, and \emph{geometrically finite} if, for sufficiently small $\epsilon>0$, its truncated convex core $C(\sigma)^{\geq\epsilon}$ is compact~\cite{MR1218098}. Here $C(\sigma)^{\geq\epsilon}$ denotes the intersection of $C(\sigma)$ with the $\epsilon$-thick part of $M_{\sigma}$. In the geometrically finite case, $\partial_{\infty}\sigma$ is a finite union of hyperbolic surfaces of finite area. Structures which are not geometrically finite are called \emph{geometrically infinite}.

\subsection{Deformation spaces~: isotopy versus homotopy}\label{sec:deformation_spaces}
Let $\cI(M)$ denote the set of hyperbolic structures on $M$, i.e.\ isotopy classes of complete hyperbolic metrics on $M$. Two metrics $\sigma_1$ and $\sigma_2$ are \emph{isotopic} if there exists a self-diffeomorphism $f$ of $M$ isotopic to the identity such that $f^*\sigma_2=\sigma_1$. Equivalently, $\cI(M)$ is the set of pairs $(N,f)$ where $N$ is a hyperbolic manifold and $f\colon M\to N$ is an orientation-preserving homeomorphism, up to the following equivalence relation~: $(N_1,f_1)\sim (N_2,f_2)$ if there exists an isometry $g\colon N_1\to N_2$ such that $g\circ f_1$ is isotopic to $f_2$.

Fix a basepoint $x_0\in M$. Given $k>1$, $\epsilon>0$, and $R>0$, we say that two points $(N_1,f_1),(N_2,f_2)\in\cI(M)$ are $(k,\epsilon,R)$-close if there exists a diffeomorphism $g\colon N_1\to N_2$ such that $g\circ f_1$ is isotopic to $f_2$ and $g$ restricted to the $R$-ball around $f_1(x_0)$ is a $(k,\epsilon)$-quasi-isometry into its image. This defines the \emph{marked pointed Gromov--Hausdorff topology}, or \emph{strong topology} on $\cI(M)$, where a basis of neighborhoods is obtained by taking $k\searrow 1$, $\epsilon\searrow 0$, and $R\nearrow +\infty$. This topology does not depend on the choice of basepoint. We denote by $\cSI(M)$ the set $\cI(M)$ endowed with this topology.

Upon replacing \emph{isotopy} by \emph{homotopy} in the above definitions, one obtains the set $\cH(M)$ of hyperbolic structures on $M$ up to homotopy and the space $\cSH(M)$, once endowed with the strong topology. The holonomy map identifies $\cSH(M)$ with a subspace of $\cSH(\pi_1(M))$, corresponding to those (conjugacy classes of) discrete and faithful representations of $\pi_1(M)$ into $\PSL(2,\C)$ that are induced by a homeomorphism from $M$ to the quotient manifold.

Let $\Mod(M)$ denote the group of isotopy classes of orientation-preserving self-homeomorphisms of $M$, and $\Mod_0(M)\subseteq\Mod(M)$ the subgroup of classes of homeomorphisms homotopic to the identity. The group $\Mod(M)$ acts on $\cI(M)$ by precomposition of the marking, $[N,f]\cdot[\phi]\coloneqq[N,f\circ\phi]$, and the quotient of $\cI(M)$ by $\Mod_0(M)$ is $\cH(M)$. Restricting homeomorphisms to the boundary of $\overline{M}$ gives a homomorphism $\Mod(M)\to\Mod(\partial M)$ (see~\cite{MR224099,MR840832} and item (3) of the theorem below) hence an action of $\Mod(M)$ on $\cT(\partial M)$ and on its augmented version $\augmTeich(\partial M)$ (see \S~\ref{sec:augmented_teichmuller_space}). The following is a recollection of classical results on the homeomorphisms of irreducible manifolds with boundary together with tameness of $3$-manifolds~\cite{MR224099,MR326737,MR825931,MR840832,agol2004tamenesshyperbolic3manifolds,MR2188131}.
\begin{theorem}\label{thm:forgetful_map}
	Let us consider the forgetful map $F\colon\cSI(M)\to\cSH(M)$ which sends the isotopy class of a hyperbolic metric to its homotopy class.
	\begin{enumerate}
		\item Each fibre of $F$ is in one-to-one correspondence with the group $\Mod_0(M)$ of isotopy classes of orientation-preserving self-homeomorphisms of $M$ homotopic to the identity.
		\item The group $\Mod_0(M)$ is generated by Dehn twists about compression disks, in particular it is trivial if $M$ has incompressible boundary.
		\item A hyperbolic structure $[N,f]\in\cSI(M)$ induces a homeomorphism $\partial f\colon \partial \overline{M}\xrightarrow{\cong}\partial\overline{N}$, well-defined up to isotopy, i.e.\ a \emph{marking} of the boundary.
	\end{enumerate}
\end{theorem}

\subsection{Cusps, parabolic locus and pared manifolds}\label{sec:cusps}
Fix a Margulis constant $\epsilon_0>0$. Components of the $\epsilon_0$-thin part are tubes about short closed geodesics or cusp neighborhoods with fundamental group $\Z$ or $\Z^2$, called rank-one or rank-two cusps, respectively \cite[\S 3.3]{MR3586015}. Rank-two cusps correspond to the components of $T$ and their presence depends only on the topology of $M$.

For a geometrically finite structure $\sigma$, the rank-one cusps determine a simple multicurve $P_{\sigma}\subset\partial M$, called its \emph{parabolic locus}~\cite{MR758464,MR2553578}. The boundary marking gives a homeomorphism
\[
    \partial_{\infty}\sigma\cong\partial M-P_{\sigma}
\]
up to isotopy, with each component of $P_{\sigma}$ giving rise to two punctures. A geometrically finite structure is \emph{minimally parabolic} if $P_{\sigma}=\emptyset$.

The pairs $(\overline{M},P)$ arising in this way are described by the notion of \emph{pared manifold}, introduced by Thurston~\cite{MR648524}, see Morgan~\cite{MR758464}. For a simple multicurve $P\subset\partial M$, let $N(P)\subset\partial M$ denote a regular neighborhood of $P$, which is a disjoint union of annuli.
\begin{definition}[Pinchable multicurves]\label{def:pinchable}
	Let $P\subset\partial M$ be a simple multicurve. The pair $(\overline{M},N(P)\cup T)$ is a \emph{pared manifold} if
	\begin{enumerate}
		\item every component of $N(P)$ is an incompressible annulus, i.e.\ every component of $P$ is homotopically non-trivial in $\overline{M}$~;
		\item every rank-two abelian subgroup of $\pi_1(\overline{M})$ is conjugate into the fundamental group of a component of $T$~;
		\item every $\pi_1$-injective map of pairs $(\bS^1\times[0,1],\bS^1\times\{0,1\})\to(\overline{M},N(P)\cup T)$ is homotopic, relative to $\bS^1\times\{0,1\}$, to a map into $N(P)\cup T$.
	\end{enumerate}
	We say that $P$ is \emph{pinchable} if $(\overline{M},N(P)\cup T)$ is a pared manifold.
\end{definition}
From condition (3) it follows that the annuli of $N(P)\cup T$ are pairwise non-parallel in $\overline{M}$~: no two components of $P$ are freely homotopic in $\overline{M}$, and no component of $P$ is freely homotopic into a component of $T$. Note that, if $P$ is pinchable, so is any sub-multicurve of $P$.

\begin{theorem}[Thurston's hyperbolization theorem for pared manifolds~\cite{MR648524,MR758464,MR2553578,MR3586015}]\label{thm:pared_hyperbolization}
	A simple multicurve $P\subset\partial M$ is the parabolic locus of some geometrically finite structure on $M$ if and only if it is pinchable.
\end{theorem}
In terms of bending laminations (Section~\ref{sec:bending_lamination}), the fact that components of $P_{\sigma}$ are pairwise non-parallel is a special case of the annulus condition of Theorem~\ref{thm:realizability_criterion}.

Let $\cMP(M)$ and $\cGF(M)$ denote the subsets of $\cI(M)$ consisting of minimally parabolic and geometrically finite structures, respectively. If $\pi_1(M)$ has no abelian subgroup of rank-two, then minimally parabolic structures are exactly convex co-compact structures, forming the subset $\cCC(M)$ of $\cI(M)$. For a pinchable multicurve $P\subset\partial M$, we denote by
\begin{equation*}
	\cGF(M;P)\subset\cGF(M)
\end{equation*}
the set of geometrically finite structures on $M$ with parabolic locus $P$, so that $\cGF(M;\emptyset)=\cMP(M)$ and $\cGF(M)$ is the disjoint union of the $\cGF(M;P)$ over pinchable multicurves $P$. We also write $\cGF(M;\emptyset,P)\coloneqq\cGF(M;\emptyset)\cup\cGF(M;P)$. For any of these spaces, we add the superscript $\nf$ to denote the subset of non-Fuchsian structures.

Unless mentioned otherwise, all these spaces are endowed with the strong topology. By the density theorem~\cite{MR2079598,MR2821565,MR3001608}, $\cMP(M)$ is dense in $\cSI(M)$.

\subsection{Quasi-isometric deformations}\label{sec:quasi_isometric_deformations}
Let $\cQI(M)$ denote $\cI(M)$ with the quasi-isometric topology: closeness is defined by diffeomorphisms respecting the isotopy markings which are $(k,\epsilon)$-quasi-isometries, with $k\searrow1$ and $\epsilon\searrow0$. The connected component containing $\sigma$ is denoted $\cQI_0(\sigma)$ and is its \emph{quasi-conformal deformation space}. Replacing isotopy markings by homotopy markings gives the spaces $\cQH(M)$ and $\cQH_0(\sigma)$.

We use the following form of simultaneous uniformization, due to the works of Ahlfors, Bers, Kra, Marden, Maskit and Sullivan. See~\cite{MR111834,MR1049503}, \cite[\S 5.3.2]{MR1638795} and \cite[\S 5.1.2]{MR3586015}. For a parabolic locus $P$, let $\Mod_0(M;P)\subseteq\Mod_0(M)$ be the subgroup preserving $P$ componentwise.

\begin{theorem}[Simultaneous uniformization]\label{thm:simultaneous_uniformization}
	Let $\sigma\in\cI(M)$ be a hyperbolic structure on $M$ and let $P$ be its parabolic locus. Sending a structure $\tau\in\cQI_0(\sigma)$ to the induced quasi-conformal deformation of the conformal boundary induces a homeomorphism
	\begin{equation}
		\cQI_0(\sigma)\longrightarrow \cT(\partial_{\infty}\sigma)
	\end{equation}
	which covers a homeomorphism
	\begin{equation}
		\cQH_0(\sigma)\longrightarrow \cT(\partial_{\infty}\sigma)/\Mod_0(M;P).
	\end{equation}
	Moreover, once endowed with their natural complex analytic structures, these homeomorphisms are biholomorphic and the maps $\cQI_0(\sigma)\to\cQH_0(\sigma)$ and $\cT(\partial_{\infty}\sigma)\to\cT(\partial_{\infty}\sigma)/\Mod_0(M;P)$ are holomorphic coverings.
\end{theorem}

The group $\Mod_0(M;P)$ acts on $\cT(\partial_{\infty}\sigma)$ via the boundary marking, that is, as Dehn twists about compressible curves. 

\begin{remark}
	On a geometrically finite quasi-conformal deformation space, the quasi-isometric and strong topologies coincide by Marden's stability theorem~\cite[\S 5.1.2]{MR3586015}. On the corresponding homotopy-marked space $\cQH_0(\sigma)$, these topologies also agree with the algebraic topology~\cite[Theorem 4.6.2]{MR3586015}.
\end{remark}
The equivariant extension results of Douady--Earle and Tukia \cite{MR857678,MR781586} give bilipschitz extensions of quasi-conformal maps, with bilipschitz constants tending to $1$ as the quasi-conformal constants tend to $1$, independently of the Kleinian group. Applied to simultaneous uniformization, this yields the following uniform estimate. Here $d_{\Teich}$ is the Teichm\"uller distance and $d_{\Lip}$ is the logarithm of the optimal bilipschitz constant among diffeomorphisms respecting the isotopy markings.
\begin{corollary}\label{cor:simultaneous_uniformization_equicontinuous}
	There exists a \emph{modulus of continuity} $\omega$, i.e.\ an increasing function
	\begin{equation*}
		\omega\colon [0,+\infty]\to [0,+\infty]
	\end{equation*}
	with $\omega(0)=0$ and $\omega(\epsilon)\to 0$ as $\epsilon\to 0$, such that, for any hyperbolizable 3-manifold $M$ and any $\sigma\in\cI(M)$, the simultaneous uniformization homeomorphism
	\begin{equation*}
		\U\colon\left(\cT(\partial_{\infty}\sigma),d_{\Teich}\right)\to\left(\cQI_0(\sigma),d_{\Lip}\right)
	\end{equation*}
	is $\omega$-uniformly continuous:
	\begin{equation*}
		d_{\Lip}(\U(m),\U(m'))\leq\omega\left(d_{\Teich}(m,m')\right)\text{ for all }m,m'\in\cT(\partial_{\infty}\sigma).
	\end{equation*}
\end{corollary}
The independence of $\omega$ from $M$ will be used in the proof of Corollary~\ref{cor:convergence_dehn_filling} to handle varying conformal boundaries along a sequence of Dehn fillings.

\subsection{Algebraic and geometric convergence}\label{sec:algebraic_and_geometric_convergence}
A sequence of representations $\rho_n\colon\pi_1(M)\to\PSL(2,\C)$ converges \emph{algebraically} to $\rho_{\infty}$ if it converges pointwise. A sequence of closed subgroups of $\PSL(2,\C)$ converges \emph{geometrically} if it converges in the Chabauty topology. We say that $\rho_n$ converges \emph{strongly} to $\rho_{\infty}$ if it converges algebraically and $\rho_n(\pi_1(M))$ converges geometrically to $\rho_{\infty}(\pi_1(M))$. For conjugacy classes, convergence means convergence of suitable representatives.

Via holonomy, the algebraic topology on $\cH(M)$ is the subspace topology that it inherits from the character variety $\cX(\pi_1(M),\PSL(2,\C))$. We denote the resulting space by $\cAH(M)$. The identity map $\cSH(M)\to\cAH(M)$ is continuous. Geometric convergence of Kleinian groups corresponds to geometric convergence of the associated framed manifolds, expressed by approximate isometries on exhausting compact subsets. See~\cite[Chapter E]{MR1219310} and \cite[Chapter 4]{MR3586015}.

The relation with the strong topology of Section~\ref{sec:deformation_spaces} is the following, see~\cite[Chapter 4]{MR3586015}~: a sequence $(\sigma_n)_{n\in\N}$ converges to $\sigma_{\infty}$ in $\cSI(M)$ if and only if the holonomies converge strongly and, for $n$ large enough, the approximate isometries $\phi_n\colon (M,\sigma_{\infty})\dasharrow (M,\sigma_n)$ can be chosen isotopic to the identity on their domain of definition, in the sense that $\phi_n\circ f_{\infty}$ is isotopic to $f_n$ on $f_{\infty}^{-1}(\mathrm{dom}\,\phi_n)$, where $f_n$ and $f_{\infty}$ denote the markings.

We shall also use strong convergence for sequences of hyperbolic manifolds whose topology may vary, following~\cite[\S 4]{MR1726737}, in which case the representations $\rho_n$ need not be faithful. This applies to the Dehn fillings considered in Section~\ref{sec:hyperbolic_dehn_filling}, where the representations are induced by the natural surjections $\pi_1(M)\epic\pi_1(M_n)$.

\subsection{Measured laminations}\label{sec:measured_laminations_and_trees}
Let $S$ be a surface of finite type endowed with a fixed ``background'' hyperbolic metric of finite area. A \emph{simple multicurve} on $S$ is a finite union of pairwise disjoint, pairwise non-isotopic, essential (i.e.\ not bounding a disk or a punctured disk) simple closed curves, considered up to isotopy. A \emph{pants decomposition} is a maximal simple multicurve. A (not necessarily simple) multicurve is \emph{filling} if every essential simple closed curve intersects it.

Let $\cGL(S)$ denote the space of geodesic laminations on $S$, with the Hausdorff topology. It is compact~\cite{MR0903850}. Let $\cML(S)$ denote the space of measured laminations on $S$, endowed with the weak-$\ast$ topology~\cite{MR3053012,MR1144770}. If $S$ has cusps, measured laminations on $S$ are assumed to have compact support. The space $\cML(S)$ is homeomorphic to an open ball of the same dimension as $\cT(S)$ and contains the space of weighted simple multicurves as a dense subset. The intersection form
\begin{equation*}
	i\colon \cML(S)\times\cML(S)\longrightarrow\R_{\geq 0}
\end{equation*}
extends continuously the geometric intersection number between simple closed curves. For $\lambda\in\cML(S)$, we denote by $\abs{\lambda}\in\cGL(S)$ its \emph{support}. A measured lamination is \emph{rational} if its support is a simple multicurve~; its \emph{closed leaves} are the components of its support which are simple closed curves and the \emph{weight} of such a leaf $c$ is the mass $\lambda(k)=i(\lambda,k)$ of a transverse arc $k$ meeting $c$ exactly once and meeting no other leaf. When $S$ is compact with boundary, we consider $\cML(S)$ as the space of measured laminations of $\Int S$ whose support is compact.
 
\subsection{Pleated surfaces}\label{sec:pleated_surfaces}
We follow Thurston~\cite{MR0903850,MR0903852}, Bonahon--Otal~\cite{MR2144972} and Lecuire~\cite{MR2464096}. A \emph{pleated map} is a continuous map $\widetilde{f}\colon\HH^2\to\HH^3$ for which there exists a geodesic lamination $\widetilde{L}\subset\HH^2$ such that $\widetilde{f}$ is a totally geodesic isometric embedding on each leaf of $\widetilde{L}$ and on each connected component of $\HH^2-\widetilde{L}$ (called the \emph{plaques} of $\widetilde{f}$). Such a map is a path-isometry and is $1$-Lipschitz. The minimal such lamination $\widetilde{L}$ is called the \emph{pleating locus} of $\widetilde{f}$ and its leaves the \emph{bending lines}.

An \emph{(abstract) pleated surface} is a triple $(\widetilde{f},G,\rho)$ where
\begin{itemize}
	\item $G$ is a discrete and torsion-free lattice in $\Isom^+(\HH^2)$,
	\item $\rho\colon G\to\Isom^{+}(\HH^3)$ is a discrete torsion-free (not necessarily faithful) representation preserving parabolic elements, and
	\item $\widetilde{f}\colon\HH^2\to\HH^3$ is a pleated map which is $G$-equivariant, i.e.\
	\begin{equation*}
		\widetilde{f}\circ a=\rho(a)\circ\widetilde{f}\quad\text{ for all }a\in G.
	\end{equation*}
\end{itemize}
The pleating locus of $\widetilde{f}$ is then $G$-invariant and descends to a geodesic lamination $L$ on the hyperbolic surface $S\coloneqq\HH^2/G$, also called the pleating locus. A \emph{pleated surface in a hyperbolic 3-manifold} $(M,\sigma)$ is a continuous map $f\colon S\to M$ from a complete hyperbolic surface $(S,h)$ of finite area which lifts to a $G$-equivariant pleated map, where $G$ is the image of the holonomy of $h$ and $\rho\coloneqq\rho_{\sigma}\circ f_{\ast}$~; conversely an abstract pleated surface descends to a pleated surface $\HH^2/G\to\HH^3/\rho(G)$.

Let $(\widetilde{f},G,\rho)$ be a pleated surface and fix an orientation on $\HH^2$. For a plaque $P$, $\widetilde{f}(P)$ is contained in a totally geodesic plane $\Pi_P$ in $\HH^3$, which divides $\HH^3$ into two closed half-spaces $H^+_P$ and $H^-_P$ (where $H^+_P$ is in the positive outward normal direction with respect to the orientation on $\Pi_P$ induced by $\HH^2$). We say that $(\widetilde{f},G,\rho)$ is \emph{convex} if the following two conditions hold, and \emph{even} if the first holds but not the second:
\begin{enumerate}
	\item There is $s\in\{+,-\}$ such that, for any plaque $P$, $\widetilde{f}(\HH^2)$ lies in $H^s_P$.
	\item The interior of
	\begin{equation*}
		C_{\widetilde{f}}\coloneqq\bigcap_{P} H^{s}_P
	\end{equation*}
	is non-empty, where $P$ ranges over all plaques.
\end{enumerate}
When $(\widetilde{f},G,\rho)$ is convex or even, $\widetilde{f}(\HH^2)$ is contained in the boundary of the convex set $C_{\widetilde{f}}$ and the amount of bending along the bending lines defines a transverse measure on the pleating locus, i.e.\ a measured lamination on $S$ called the \emph{bending (measured) lamination} of the pleated surface~\cite{MR0903852}. The closed leaves of the bending lamination have weight in $(0,\pi]$, and a leaf has weight $\pi$ exactly when the two adjacent plaques are folded onto each other, in which case $C_{\widetilde{f}}$ has empty interior and the surface is even. Such leaves are not to be confused with the later use of ``weight-$\pi$ closed leaves'' that are associated to rank-one cusps.

\begin{definition}\label{def:convergence_abstract_pleated_surfaces}
	A sequence of pleated surfaces $\{{(\widetilde{f}_n,G_n,\rho_n)}\}_{n\in\N}$ converges to a pleated surface $(\widetilde{f},G,\rho)$ if the following three conditions hold~:
	\begin{enumerate}
		\item $G_n$ converges geometrically to $G$~;
		\item for any sequence $a_n\to a$ with $a_n\in G_n$, $n\in\N$, and $a\in G$, $\rho_n(a_n)$ converges to $\rho(a)$~;
		\item $\widetilde{f}_n$ converges compactly to $\widetilde{f}$.
	\end{enumerate}
\end{definition}

If we restrict our attention to pleated surfaces $(f,G,\rho)$ of a given topological type $S$, then $G$ can be replaced by a discrete and faithful representation $r\colon\pi_1(S)\to\Isom^+(\HH^2)$. The first two conditions in Definition~\ref{def:convergence_abstract_pleated_surfaces} then reduce to: ``$r_n(\gamma)\to r(\gamma)$ and $\rho_n(r_n(\gamma))\to\rho(r(\gamma))$ for all $\gamma\in\pi_1(S)$''.

\begin{proposition}[Lecuire~\cite{MR2464096}]\label{prop:Lecuire_continuity_bending_lamination_if_pleated_surfaces_converge}
	Let $\{(\widetilde{f}_n,\Gamma_n,\rho_n)\}_{n\in\N}$ be a sequence of convex pleated surfaces converging to a convex pleated surface $(\widetilde{f},\Gamma,\rho)$. Let $(\widetilde{\lambda}_n)_{n\in\N}$ and $\widetilde{\lambda}$ be their bending laminations. Then $\widetilde\lambda_n\to\widetilde\lambda$.
\end{proposition}

\begin{lemma}[Bonahon--Otal~\cite{MR2144972}, Lecuire~\cite{MR2464096}]\label{lem:convex_or_even_pleated_surfaces_closed}
	The set of pleated surfaces that are convex or even is closed in the set of pleated surfaces.
\end{lemma}

\subsection{Bending lamination}\label{sec:bending_lamination}
For a non-Fuchsian structure $\sigma$, the convex-core boundary is a convex pleated surface and carries a bending lamination, which can be seen as living on $\partial_{\infty}\sigma$. This defines the bending map
\begin{equation*}
    \bb\colon\cQI_0(\sigma)\to\cML(\partial_{\infty}\sigma).
\end{equation*}
For geometrically finite $\sigma$, this map is continuous
and tangentiable~\cite{MR1678469}.

For geometrically finite structures on $M$, assigning weight $\pi$ to each component of the parabolic locus gives a measured lamination on $\partial M$. We thus obtain
\begin{equation*}
    \bb\colon\cGF^{\nf}(M)\to\cML(\partial M),
\end{equation*}
whose image is characterized as follows.

\begin{theorem}[Realizability criterion, Bonahon--Otal~\cite{MR2144972}, Lecuire~\cite{MR2207784}]\label{thm:realizability_criterion}
	Let $M$ be a hyperbolizable 3-manifold. A measured lamination $\lambda\in\cML(\partial M)$ is the bending lamination of a non-Fuchsian geometrically finite structure on $M$ if and only if it satisfies the following conditions~:
	\begin{enumerate}
		\item $\lambda$ has no closed leaf of weight strictly greater than $\pi$~;
		\item \emph{Annulus condition:} there exists $\eta>0$ such that, for any essential annulus $A$ in $M$, we have $i(\partial A,\lambda)\geq\eta$~;
		\item \emph{Disk condition:} for any essential disk $D$ in $M$, we have $i(\partial D,\lambda)>2\pi$.
	\end{enumerate}
\end{theorem}
Moreover, when $\lambda$ is rational, Bonahon and Otal also showed that the geometrically finite structure realizing $\lambda$ is unique~\cite{MR2144972}.

\begin{remark}\label{rem:bending_fuchsian}
	If $M$ is homeomorphic to $S\times\R$ for a complete hyperbolizable surface $S$, then it admits Fuchsian structures. If $S$ is closed, then the convex core of a Fuchsian structure is a totally geodesic copy of $S$ inside $M$. In that case, its bending lamination is the measured lamination zero. If $S$ is not closed, then a Fuchsian structure on $M$ is either:
	\begin{itemize}
		\item convex co-compact, in which case its convex core is a totally geodesic surface with boundary $\overline{F}$ and its ideal boundary is homeomorphic to the double $D\overline{F}\coloneqq \overline{F}\cup\overline{F}^{-}/\partial\overline{F}\sim\partial\overline{F}^{-}$ of $\overline{F}$, where $\overline{F}^{-}$ denotes $\overline{F}$ with the opposite orientation. In that case, one can consider its bending lamination as the element of $\cML(D\overline{F})$ supported on $\partial\overline{F}$ with weight $\pi$ on each component~;
		\item geometrically finite with rank-one cusps, when certain (or all) boundary components of $\overline{F}$ get pinched until they become cusps.
	\end{itemize}
	In particular, Fuchsian structures cannot be distinguished by their bending lamination. See Section~\ref{sec:punctured_surface_groups} for a discussion of quasi-Fuchsian structures over \emph{punctured surfaces}.
\end{remark}

\begin{definition}[Realizable laminations]
	Let $M$ be a hyperbolizable 3-manifold. A measured lamination $\lambda\in\cML(\partial M)$ is said to be \emph{realizable} if it is the bending lamination of a non-Fuchsian geometrically finite structure on $M$, or equivalently if it satisfies the three conditions of the previous theorem.
	
	Let $\cP(M)$ be the subset of $\cML(\partial M)$ consisting of realizable laminations, and $\cP_{\nc}(M)\subset\cP(M)$ the subset of realizable laminations without closed leaf of weight $\geq\pi$.
\end{definition}

As explained by Bonahon--Otal and Lecuire, $\bb\colon\cGF(M)\to\cML(\partial M)$ is not continuous. Lecuire shows in \cite{MR2464096,lecuire2025propernessbendingmap} that $\cP(M)$ should be endowed with the \emph{tubular topology} to make $\bb$ continuous, see Definition~\ref{def:tubular_topology}. Let $\cBL(M)$ denote the set $\cP(M)$ endowed with the tubular topology.

\begin{theorem}[Lecuire~\cite{MR2464096,lecuire2025propernessbendingmap}]\label{thm:Lecuire_continuous_proper_on_non_Fuchsian_locus}
	The bending map $\bb\colon\cGF^{\nf}(M)\to\cBL(M)$ is continuous and proper.
\end{theorem}

If we restrict to convex co-compact structures, the bending map is known to be injective~\cite{dularschlenker2024pleating}. Combined with continuity and properness, we obtain:
\begin{theorem}[Bonahon--Otal, Lecuire, Dular--Schlenker]\label{thm:injectivity_convex_cocompact}
	Let $M$ be a hyperbolizable 3-manifold such that $\pi_1(M)$ has no abelian subgroup of rank-two. Then the bending map $\bb\colon\cCC^{\nf}(M)\to\cP_{\nc}(M)$ is a homeomorphism.
\end{theorem}

\subsection{Real-analyticity}\label{sec:real_analyticity}
Using Bonahon's holomorphic shear-bend coordinates~\cite{MR1413855} (with cusped surfaces treated in \cite[\S 12.3]{MR1413855}) and the fact that quasi-conformal deformation spaces are locally biholomorphic to complex submanifolds of the character variety, see~\cite[Theorems 7.53-55]{MR1638795},\cite{MR1208313}, we have the following result~\cite{dularschlenker2024pleating}.
\begin{proposition}\label{prop:real_analycity}
	Let $M$ be a hyperbolizable 3-manifold and $\sigma\in\cI(M)$ a hyperbolic structure on $M$. Let $S$ be a connected component of $\partial_{\infty}\sigma$ and $\lambda\in\cML(S)$ a measured lamination. Then the subspace of $\cQI_0(\sigma)$ consisting of those structures for which $\lambda$ is the bending lamination of the boundary component of the convex core facing $S$ is a real-analytic subspace. In particular, fibres of the bending map $\bb\colon\cQI_0(\sigma)\to\cML(\partial_{\infty}\sigma)$ are real-analytic spaces.
\end{proposition}

This proposition will be essential for our purpose via the following results. Recall that a metrizable space $X$ is an \emph{absolute neighborhood retract} (ANR) if, whenever $X$ is embedded as a closed subset of a metrizable space $Y$, it is a retract of some neighborhood of itself in $Y$~\cite{MR0872468}.

\begin{theorem}[Borel--Haefliger~\cite{MR149503}, Sullivan~\cite{MR278333}]\label{thm:real_analytic_spaces_have_fundamental_class}
	A compact real-analytic space $X$ has a fundamental class in homology with $\Z/2\Z$ coefficients. In particular, $H_{d}(X;\Z/2\Z)\neq 0$ if $X$ has dimension $d$.
\end{theorem}

\begin{lemma}\label{lem:real_analytic_spaces_are_ANRs}
	Compact real-analytic spaces are absolute neighborhood retracts.
\end{lemma}
\begin{proof}
	A compact real-analytic space is a compact semi-analytic subset of a real-analytic manifold, hence it admits a (finite) triangulation by a theorem of \L{}ojasiewicz~\cite{MR173265}. A compact polyhedron is an ANR, see e.g.~\cite[Corollary 8A]{MR0872468}.
\end{proof}

\begin{corollary}\label{cor:contractible_compact_real_analytic_spaces_are_singletons}
	Contractible compact real-analytic spaces are singletons.
\end{corollary}
\begin{proof}
	A compact real-analytic space $X$ has a fundamental class in homology with coefficients in $\Z/2\Z$ (by Theorem~\ref{thm:real_analytic_spaces_have_fundamental_class}). But contractibility forces this fundamental class to live in degree zero, hence $X$ is a discrete union of points. By contractibility again, $X$ must be a singleton.
\end{proof}

\section{Contractibility of boundary fibres}\label{sec:boundary_of_homeomorphisms}
To prove injectivity of the bending map in the convex co-compact case~\cite{dularschlenker2024pleating}, we used the following theorem of Finney~\cite{MR0224087}.
\begin{theorem}[Finney]\label{thm:finney}
	Let $M,M'$ be topological manifolds of the same dimension. Let $f_n\colon M\to M'$ be a sequence of continuous maps converging continuously (equivalently, compactly) to a continuous map $f_{\infty}\colon M\to M'$. Assume that each $f_n$ is injective. If a fibre of $f_{\infty}$ is compact and a retract of a relatively compact neighborhood, then it is contractible.
\end{theorem}

For our purpose, we need the following two results: the first is a variant with slightly weakened conditions, while the second replaces sequences of functions with families of functions indexed by some space $T$ together with a point at infinity.

\begin{corollary}\label{cor:finney_variant}
	Let $A,B$ be topological manifolds of the same dimension. Let $f_n\colon A\to B$ be a sequence of continuous maps converging continuously (equivalently, compactly) to a continuous map $f_{\infty}\colon A\to B$. Let $b\in f_{\infty}(A)\subseteq B$ and let $F\coloneqq f_{\infty}^{-1}(b)$ be its fibre. Assume that the following conditions are satisfied:
	\begin{enumerate}
		\item There exists $n_0\in\N$ and a neighborhood $U$ of $F$ in $A$ such that the restrictions $f_n\restr{U}$ are injective for all $n\geq n_0$~;
		\item The fibre $F$ is compact and an absolute neighborhood retract.
	\end{enumerate}
	Then $F$ is contractible.
\end{corollary}
\begin{proof}
	Apply Finney's theorem to the restricted sequence $(f_n\restr{U})_{n\geq n_0}$.
\end{proof}

\begin{theorem}[Contractibility of the fibres of the boundary of a homeomorphism]\label{thm:fibres_contractible_boundary_of_homeo}
	Let $A=A_{\circ}\cup A_{\infty}$ be a first-countable Hausdorff space such that $A_{\circ}\cap A_{\infty}=\emptyset$, $A_{\infty}$ is closed in $A$ and $A_{\infty}\subset \overline{A_{\circ}}$. Let $B=B_{\circ}\cup B_{\infty}$ satisfy the same condition.

	Let $\Phi\colon (A;A_{\circ},A_{\infty})\to (B;B_{\circ},B_{\infty})$ be a continuous map of triples, i.e.\ $\Phi\colon A\to B$ is a continuous map satisfying $\Phi(A_{\circ})\subseteq B_{\circ}$ and $\Phi(A_{\infty})\subseteq B_{\infty}$. Assume that the following conditions are satisfied~:
	\begin{enumerate}
		\item\label{it:properness} $\Phi$ is proper onto its image~;
		\item $\Phi\restr{A_{\circ}}$ is an open embedding~;
		\item\label{it:local_product} The space $A$ is locally a product near $A_{\infty}$~: for any compact subset $K\subset A_{\infty}$, there is a neighborhood $U$ in $A$, a first-countable Hausdorff space $T\cup\{\infty\}$ where $\infty\in\overline{T}$ and a homeomorphism of triples
		\begin{equation*}
			j\colon (U_{\infty}\times (T\cup\{\infty\});U_{\infty}\times T,U_{\infty}\times \{\infty\})\to (U;U_{\circ},U_{\infty}),
		\end{equation*}
		where $U_{\infty}=U\cap A_{\infty}$, $U_{\circ}=U\cap A_{\circ}$ and $j(x,\infty)=x$ for all $x\in U_{\infty}$~;
		\item\label{it:contractibility_at_infinity} $\Phi$ has contractible image at infinity~: any point $y\in\Phi(A_{\infty})$ has a basis of neighborhoods $V$ such that $\Phi(A_{\circ})\cap V$ is contractible~;
		\item\label{it:retract_property} Nonempty fibres of $\Phi\restr{A_{\infty}}\colon A_{\infty}\to B_{\infty}$ are retracts of relatively compact neighborhoods in $A_{\infty}$.
	\end{enumerate}
	Then nonempty fibres of $\Phi\restr{A_{\infty}}\colon A_{\infty}\to B_{\infty}$ are contractible.
\end{theorem}
\begin{proof}
	Let $y$ be a point of $B_{\infty}$ whose fibre $F\coloneqq\Phi^{-1}(y)\subseteq A_{\infty}$ is nonempty. By condition~\eqref{it:retract_property}, $F$ has a relatively compact neighborhood $\Omega\subset A_{\infty}$ which retracts onto $F$. Let $r\colon\Omega\epic F$ be a retraction, i.e.\ $r$ is continuous and $r\restr{F}=\id_F$. The boundary $\partial\Omega$ of $\Omega$ in $A_{\infty}$ is compact and disjoint from $F$.

	Let us assume that $F$ is not contractible. We will show that $\Phi$ sends a point of $\partial\Omega$ to $y$, which is impossible.

	Let $\cV$ be a nested (countable) basis of neighborhoods of $y$ in $B$ such that $\Phi(A_{\circ})\cap V$ is contractible for each $V\in\cV$, as in condition~\eqref{it:contractibility_at_infinity}.
	\begin{claim}
		Let $U$ be a neighborhood of $F$ in $A$. Up to passing to a subbasis, we have $\Phi^{-1}(V)\subseteq U$ for all $V\in\cV$.
	\end{claim}
	\begin{proof}[Proof of the claim]
		Otherwise, there is a sequence $(V_n)_{n\in\N}$ in $\cV$ with $\bigcap_{n\in\N}V_n=\{y\}$ and a sequence $(a_n)_{n\in\N}$ such that $a_n\in\Phi^{-1}(V_n)-U$ for each $n\in\N$. In particular, $\Phi(a_n)\in V_n$ for all $n$, which implies that $\Phi(a_n)\to y$. By properness of $\Phi$, up to a subsequence, $(a_n)_{n\in\N}$ converges to a point $a$ in $F=\Phi^{-1}(y)$. But then $(a_n)_{n\in\N}$ must eventually lie inside $U$, a contradiction.
	\end{proof}

	Let $U$ be a neighborhood of $\overline{\Omega}$ in $A$, $j$ and $T\cup\{\infty\}$ be as in condition~\eqref{it:local_product}. To simplify notation, we identify $U$ with $U_{\infty}\times (T\cup\{\infty\})$ and $\Phi$ with $\Phi\circ j$ from now on, i.e.\ we consider $\Phi$ as a map from $U_{\infty}\times (T\cup\{\infty\})$ to $B$.
	
	By the above claim, we can assume that $\Phi^{-1}(V)\subseteq U$ for all $V\in\cV$.

	\begin{claim}\label{claim:nullhomotopy_meets_boundary}
		Let $V\in\cV$. The intersection $\Phi^{-1}(V)\cap (\partial\Omega\times T)$ is nonempty.
	\end{claim}
	\begin{proof}[Proof of the claim]
		By compactness of $F$ and by the product structure of $U$, there exists $\tau\in T$ such that $F\times\{\tau\}$ is contained in $\Phi^{-1}(V)$, hence $F\times\{\tau\}\subset\Phi^{-1}(V)\cap U_{\circ}$.

		Since $\Phi_{\circ}\coloneqq\Phi\restr{A_{\circ}}$ is an open embedding, $\Phi^{-1}(V)\cap U_{\circ}=\Phi^{-1}(V)\cap A_{\circ}=\Phi_{\circ}^{-1}(V\cap\Phi(A_{\circ}))$ is contractible and there exists a nullhomotopy
		\begin{equation*}
			h\colon F\times [0,1]\to \Phi^{-1}(V)\cap U_{\circ}
		\end{equation*}
		where $h(x,0)=(x,\tau)$ for each $x\in F$ and $h(-,1)$ is constant.

		We now show that the image of $h$ meets $\partial\Omega\times T$, which will conclude the proof of the claim (see Figure~\ref{fig:boundary_of_homeo}). Assume it does not. We have
		\begin{equation*}
			U_{\circ}-(\partial\Omega\times T)=(U_{\infty}-\partial\Omega)\times T=(\Omega\times T)\cup ((U_{\infty}-\overline{\Omega})\times T),
		\end{equation*}
		which is a decomposition of the open set $U_{\circ}-(\partial\Omega\times T)$ as the union of two disjoint open subsets. Since the image of $h$ is connected (as $h(-,1)$ is constant), it must lie entirely inside one or the other of these two subsets. But $h(F,0)=F\times\{\tau\}\subset\Omega\times T$, hence the image of $h$ must be entirely contained in $\Omega\times T$.

		We can now consider the continuous map defined as the composition
		\begin{equation*}
			H\colon F\times [0,1]\xrightarrow{h} \Omega\times T\xrightarrow{\proj_{\Omega}}\Omega\xrightarrow{r} F,
		\end{equation*}
		which is a homotopy from $H(-,0)=\id_F$ to a constant map $H(-,1)$. This contradicts our assumption that $F$ is not contractible and proves the claim.
	\end{proof}

	Now, let $(V_n)_{n\in\N}$ be a sequence in $\cV$ with $\bigcap_{n\in\N}V_n=\{y\}$. By the above claim, for each $n\in\N$ there exists $(x_n,\tau_n)\in(\partial\Omega\times T)\cap\Phi^{-1}(V_n)$. Thus $\Phi(x_n,\tau_n)\to y$. By properness of $\Phi$ and up to a subsequence, this implies that $(x_n,\tau_n)\to (x,\tau)$ for some $(x,\tau)\in\Phi^{-1}(y)=F\times\{\infty\}$. In particular $x_n\to x$. But $x_n\in\partial\Omega$ for all $n\in\N$ and $\partial\Omega$ is compact, hence $x\in F\cap\partial\Omega=\emptyset$, which is impossible.
\end{proof}

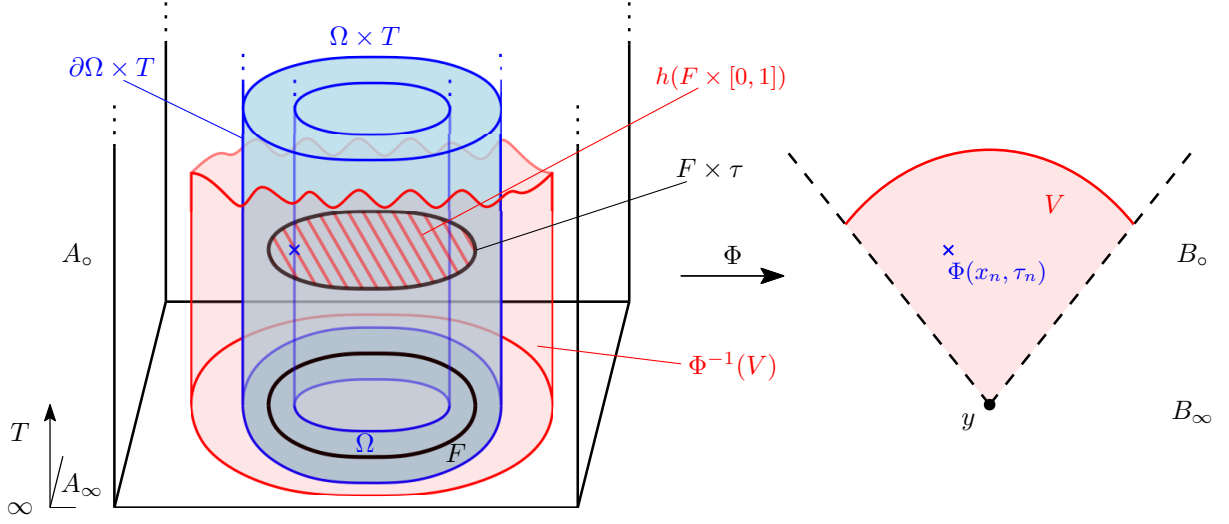
\begin{figure}[ht]
		\centering
		\begin{tikzpicture}
			\node[anchor=south west,inner sep=0] (image) at (0,0) {\includegraphics[width=0.95\textwidth]{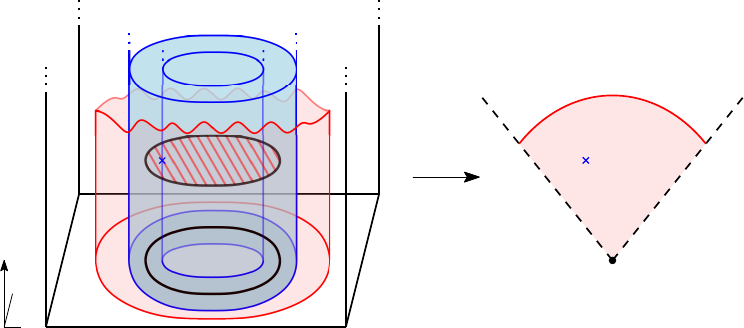}};
			\begin{scope}[x={(image.south east)},y={(image.north west)}]
			\node[black] at (0.58,0.67) {$F\times \tau$};
			\node[black] at (0.805,0.17) {$y$};
			\node[black] at (0.03,0.5) {$A_{\circ}$};
			\node[black] at (0.034,0.045) {$A_{\infty}$};
			\node[black] at (-0.02,0.15) {$T$};
			\node[black] at (-0.02,0) {$\infty$};
			\node[black] at (1,0.5) {$B_{\circ}$};
			\node[black] at (1,0.19) {$B_{\infty}$};
			\node[blue] at (0.28,0.93) {$\Omega\times T$};
			\node[blue] at (0.06,0.87) {$\partial\Omega\times T$};
			\node[black] at (0.359,0.11) {$F$};
			\node[blue] at (0.28,0.128) {$\Omega$};
			\node[black] at (0.6,0.5) {$\Phi$};
			\node[red] at (0.59,0.85) {\small $h(F\times [0,1])$};
			\node[red] at (0.6,0.28) {$\Phi^{-1}(V)$};
			\node[red] at (0.88,0.6) {$V$};
			\node[blue] at (0.83,0.465) {\small $\Phi(x_n,\tau_n)$};
			
			\draw[black] (0.377,0.51) -- (0.56,0.645);
			\draw[blue] (0.07,0.845) -- (0.174,0.73);
			\draw[red] (0.33,0.54) -- (0.57,0.83);
			\draw[red] (0.43,0.32) -- (0.555,0.285);
			\end{scope}
		\end{tikzpicture}
		\caption{Illustration of the proof of Claim~\ref{claim:nullhomotopy_meets_boundary} in the proof of Theorem~\ref{thm:fibres_contractible_boundary_of_homeo}. The nullhomotopy of $F\times\{\tau\}$ inside $\Phi^{-1}(V)\cap A_{\circ}$ must meet $\partial\Omega\times T$. This point of intersection $(x_n,\tau_n)$ (blue cross) is sent by $\Phi$ to a point $\Phi(x_n,\tau_n)$ in $V$.}\label{fig:boundary_of_homeo}
\end{figure}

The condition that $\Phi$ has contractible image at infinity closely resembles the \emph{cell-like} (or \emph{$UV^{\infty}$}) property from \emph{shape theory} (see e.g.\ Lacher~\cite{MR251714} or Daverman~\cite{MR0872468}). For ANRs, being cell-like is equivalent to being contractible. While Theorem~\ref{thm:fibres_contractible_boundary_of_homeo} might follow from general results (e.g.\ combining the local product and contractibility at infinity conditions to prove that the fibres of $\Phi\restr{A_{\infty}}$ are cell-like), the formulation above directly applies to our situation and provides a self-contained proof.

\section{Minimally parabolic structures}\label{sec:minimally_parabolic_structures}
The goal of this section is to prove that the bending map is a homeomorphism onto its image for minimally parabolic hyperbolic structures on manifolds with rank-two cusps (Theorem~\ref{thmA} of the introduction).
\begin{theorem}\label{thm:minimally_parabolic}
	Let $M$ be a hyperbolizable 3-manifold with rank-two cusps. The bending map
	\begin{equation*}
		\bb\colon\cMP(M)\to\cML(\partial M)	
	\end{equation*}
	is a homeomorphism onto its image.
\end{theorem}
The image of the bending map is the subset $\cP_{\nc}(M)$ described by the realization theorem of Bonahon--Otal and Lecuire (Theorem~\ref{thm:realizability_criterion}), consisting of measured laminations on $\partial M$ without closed leaves of weight $\geq\pi$ and satisfying the annulus and disk conditions.

Let $M$ be a hyperbolizable manifold with rank-two cusps. Let $\rho\colon\pi_1(M)\to\Isom(\HH^3)$ be a minimally parabolic structure on it and $\Gamma\coloneqq\rho(\pi_1(M))$ the associated Kleinian group. Let $C_1,\dots,C_k$ be its torus cusp neighborhoods, together with choices of meridians and longitudes (i.e.\ a choice of markings $\Z^2\cong\pi_1(C_i)$). Given a tuple of \emph{slopes}, i.e.\ pairs of coprime integers $\mathbf{s}=\{(p_i,q_i)\}^k_{i=1}$ (corresponding to primitive curves on the torus cusp neighborhoods $\partial C_i$), denote by $M(\mathbf{s})$ the result of \emph{Dehn filling} $M-(C_1\cup\dots\cup C_k)$ along each torus boundary component $\partial C_i$ with slope $(p_i,q_i)$. Except for finitely many \emph{exceptional slopes} (finitely many in each entry), the resulting 3-manifold $M(\mathbf{s})$ admits a hyperbolic structure~\cite{MR648524,MR2680207,MR4468992}. The inclusion $M\monic M(\mathbf{s})$ induces an isotopy class of homeomorphisms $\partial M\cong\partial M(\mathbf{s})$. See also~\cite{MR648524,comar1996hyperbolic,MR2079598,MR4466646,MR4468992} for more about hyperbolic Dehn filling.

Pick a sequence $(\mathbf{s}^{n})_{n\in\N}$ of non-exceptional slopes going to infinity and denote by $M_n$ the resulting sequence of topological Dehn fillings. For each $n\in\N$, there is a natural inclusion $j_n\colon M\monic M_n$ that induces an identification $\partial j_n\colon\partial M\cong\partial M_n$. Let
\begin{equation*}
	F_n\colon\cMP(M)\rightarrow\cCC(M_n)
\end{equation*}
be the map that sends a minimally parabolic structure on $M$ to the convex co-compact structure on $M_n$ with identical conformal boundary. For each $n\in\N$, let $\bb_n\coloneqq\bb\circ F_n$.
\[\begin{tikzcd}
	& {\cT(\partial M)} \\
	{\cMP(M)} && {\cCC(M_n)} \\
	& {\cML(\partial M)}
	\arrow["\cong", from=2-1, to=1-2]
	\arrow["{F_n}", from=2-1, to=2-3]
	\arrow["\bb"',bend right=10pt, from=2-1, to=3-2]
	\arrow["{\bb_n}"{pos=0.9},bend left=25pt, shift right, from=2-1, to=3-2]
	\arrow["\cong"', from=2-3, to=1-2]
	\arrow["\bb", from=2-3, to=3-2]
\end{tikzcd}\]

Using the geometrically finite version of the hyperbolic Dehn filling theorem combined with a continuity result of Lecuire~\cite{MR2464096}, we show in Proposition~\ref{prop:continuous_convergence_drilling} that $\bb_n$ converges continuously to $\bb$.

Injectivity of $\bb$ then follows as in~\cite{dularschlenker2024pleating}, by Finney's theorem (its variant Corollary~\ref{cor:finney_variant}) and the convex co-compact case~\cite{dularschlenker2024pleating}. See Section~\ref{sec:conclusion_proof_minimally_parabolic} for the conclusion of the proof.

\subsection{Hyperbolic drilling and Dehn filling}\label{sec:hyperbolic_dehn_filling}
We shall need the Brock--Bromberg drilling theorem~\cite[Theorem 6.2]{MR2079598}, as stated in~\cite[Theorem 1.1]{MR4468992}.
\begin{theorem}[Drilling theorem, Brock--Bromberg~\cite{MR2079598}]\label{thm:brock_bromberg_drilling}
	Given $J>1$ and $\epsilon>0$ with $\epsilon\leq\epsilon_0$, there exists $\ell_0(\epsilon,J)>0$ such that the following holds: If $N$ is a 3-manifold with a minimally parabolic structure $\sigma$ and $L\subset N$ is a geodesic link whose total length is less than $\ell_0(\epsilon,J)$, then $N-L$ admits a minimally parabolic structure $\sigma'$ with the same conformal structure on the ideal boundary as $\sigma$, and the inclusion
	\begin{equation*}
		j\colon (N-L,\sigma')\monic (N,\sigma)
	\end{equation*}
	restricts to a $J$-bilipschitz diffeomorphism on the complement of $\epsilon$-thin tubes about $L$.
\end{theorem}

\begin{corollary}\label{cor:convergence_dehn_filling}
	Let $M$ be a hyperbolizable 3-manifold with rank-two cusps $C_1,\dots,C_k$. Let ${(\mathbf{s}^n)}_{n\in\N}$ be a sequence of tuples of slopes on $C_1,\dots,C_k$ going to infinity, i.e.\ such that for each $i=1,\dots,k$, the sequence of slopes ${(s_i^n)}_{n\in\N}$ on $C_i$ goes to infinity. Let $M_n$ be the result of Dehn filling $M$ along $\mathbf{s}^n$, $j_n\colon M\monic M_n$ the induced inclusion and $\partial M\cong\partial M_n$ the induced homeomorphism of the ideal boundaries.

	For each $n\in\N$, let $\sigma_n\in\cCC(M_n)$ be a convex co-compact structure on $M_n$, with conformal boundary $m_n\in\cT(\partial M)$. Let $\sigma\in\cMP(M)$ be a minimally parabolic structure on $M$, with conformal boundary $m\in\cT(\partial M)$.
	
	If $m_n\to m$ in $\cT(\partial M)$, then:
	\begin{enumerate}
		\item the inclusions $j_n\colon (M,\sigma)\to (M_n,\sigma_n)$ are isotopic to embeddings, diffeomorphic onto their images, whose restrictions to the $\epsilon_n$-thick part of $(M,\sigma)$ are $J_n$-bilipschitz, where $\epsilon_n\searrow 0$ and $J_n\searrow 1$,
		\item their holonomies
		\begin{equation*}
			\rho_n\colon\pi_1(M)\epic\pi_1(M_n)\to\Isom^+(\HH^3)
		\end{equation*}
		converge strongly to the holonomy $\rho\colon\pi_1(M)\to\Isom^+(\HH^3)$ of $\sigma$, i.e.\ $\rho_n\to\rho$ pointwise and $\rho_n(\pi_1(M))\to\rho(\pi_1(M))$ in the Chabauty topology,
	\end{enumerate}
	i.e.\ $\sigma_n\to\sigma$ strongly in the sense of Section~\ref{sec:algebraic_and_geometric_convergence}.
\end{corollary}
\begin{proof}
	For each $n\in\N$, let $\sigma'_n\in\cCC(M_n)$ be the Dehn filling of $M$ with slope $\mathbf{s}^n$, i.e.\ the convex co-compact structure on $M_n$ with conformal boundary $m\in\cT(\partial M)$. In $M_n$, the cusps $C_1,\dots,C_k$ of $M$ are filled with core curves $c^n_1,\dots,c^n_k\subset M_n$, whose $\sigma'_n$-length goes to zero as $n\to +\infty$, since the slopes go to infinity.

	By the drilling Theorem~\ref{thm:brock_bromberg_drilling}, the inclusions $j_n\colon (M,\sigma)\to (M_n,\sigma'_n)$ become arbitrarily close to isometries on the complement of arbitrarily small cusp neighborhoods.

	By assumption, $d_{\Teich}(m_n,m)\to 0$ where $m_n$ and $m$ are the conformal boundaries of $\sigma_n$ and $\sigma'_n$, respectively. It then follows from uniform equicontinuity of simultaneous uniformization (Corollary~\ref{cor:simultaneous_uniformization_equicontinuous}) that
	\begin{equation*}
		d_{\Lip}(\sigma_n,\sigma'_n)\leq\omega(d_{\Teich}(m_n,m))\to 0,
	\end{equation*}
	where $\omega$ is independent of the manifolds $M_n$.

	Therefore, the inclusions $j_n\colon (M,\sigma)\to (M_n,\sigma_n)$ also become arbitrarily close to isometries on the complement of arbitrarily small cusp neighborhoods. In other words, the hyperbolic manifolds $(M_n,\sigma_n)$ converge geometrically to $(M,\sigma)$.

	By~\cite[Theorem E.1.13]{MR1219310} (see also~\cite[\S 4]{MR1726737}), this implies strong convergence of the holonomies (see Sections~\ref{sec:deformation_spaces} and \ref{sec:algebraic_and_geometric_convergence}).
\end{proof}

\subsection{Polygonal approximations and local injectivity}\label{sec:polygonal_approximations}
Let $\widetilde{f}\colon\HH^2\to\HH^3$ be a convex pleated map and $w$ a path in $\widetilde{f}(\HH^2)$, which we view as a path in the boundary of the convex set $C_{\widetilde{f}}$. Following Keen--Series~\cite{MR1331998}, a \emph{polygonal approximation} to $w$ is a finite sequence $\cP=\{(x_i,P_i)\}_{i=0}^k$ where the $x_i\in w$ are linearly ordered along $w$, with $x_0$ and $x_k$ the endpoints of $w$, and where $P_i$ is a support plane of $C_{\widetilde{f}}$ at $x_i$, such that $P_i\cap P_{i-1}\neq\emptyset$ for $i=1,\dots,k$. Let $\theta_i\coloneqq\theta(P_{i-1},P_i)\in[0,\pi)$ be the exterior dihedral angle between $P_{i-1}$ and $P_i$, and $d_i\coloneqq d_w(x_{i-1},x_i)$ the distance from $x_{i-1}$ to $x_i$ along $w$, for the path metric on $\widetilde{f}(\HH^2)$. The approximation is an \emph{$(\epsilon,\delta)$-approximation} if
\begin{equation*}
	\max_{i}\theta_i<\epsilon\quad\text{and}\quad\max_{i}d_i<\delta.
\end{equation*}
For every $\epsilon,\delta>0$, every path $w$ admits an $(\epsilon,\delta)$-approximation~\cite[Lemma 3.2]{MR2464096}.

\begin{proposition}[Error estimate, Keen--Series~{\cite[Proposition 4.8]{MR1331998}}]\label{prop:error_estimate_keen_series}
	There exists a universal constant $K>0$ and a function $s(\epsilon)$, $0<s(\epsilon)<1$, such that if $\cP$ is an $(\epsilon,s(\epsilon))$-approximation to a path $w$ in a convex pleated surface, where $0<\epsilon<\pi/2$, then
	\begin{equation*}
		\abs{\sum_{\cP}\theta_i-i(\beta,w)}<K\epsilon\ell(w)\quad\text{and}\quad\abs{\sum_{\cP}d_i-\ell(w)}<K\epsilon\ell(w),
	\end{equation*}
	where $\ell(w)$ is the length of $w$ for the induced metric on the pleated surface and $\beta$ is its bending lamination.
\end{proposition}
Keen and Series state the estimate for the boundary of the convex core of a hyperbolic 3-manifold~; the proof is local and applies readily to any convex pleated map.

The following lemma shall be used in Proposition~\ref{prop:continuous_convergence_drilling}.
\begin{lemma}\label{lem:convex_pleated_surface_locally_injective}
	Let $g\colon\HH^2\to\HH^3$ be a convex pleated map, with pleating locus $L$ and bending measure $\mu$. For any geodesic path $k$ in $\HH^2$, if $i(\mu,k)<\pi$ then the restriction $g\restr{k}$ is injective.
\end{lemma}
\begin{proof}
	Let $a$ and $b$ be the endpoints of $k$. It suffices to show that $g(a)\neq g(b)$. Indeed, injectivity then follows by restricting to all subpaths of $k$.
	
	Let $\epsilon>0$ be such that $K\epsilon\ell(k)<\frac{1}{2}(\pi-i(\mu,k))$. By Keen--Series error estimate (Proposition~\ref{prop:error_estimate_keen_series}), there exists a $(\epsilon,s(\epsilon))$-approximation $\cP=\{(x_i,P_i)\}_{i=0}^n$ to the path $k$, with $x_0=a$ and $x_n=b$, such that
	\begin{equation}\label{eq:proof_local_injectivity_of_slightly_bent_arcs}
		\sum_{i=1}^n\theta(P_{i-1},P_i)<i(\mu,k)+K\epsilon\ell(k)<\pi.
	\end{equation}
	Passing to a sub-approximation, we can assume that consecutive planes are distinct. In particular, $\theta(P_{i-1},P_i)>0$ for all $i=1,\dots,n$.

	Let $p$ be the polygonal path in $\HH^3$ with consecutive vertices $p_0,p_1,\dots,p_n,p_{n+1}$ where $g(a)=p_0$ and $g(b)=p_{n+1}$ and, for each $i\in\{1,\dots,n\}$,
	\begin{enumerate}
		\item $p_{i}\in P_{i-1}\cap P_i$, and
		\item the consecutive segments $[p_{i-1},p_i]$ and $[p_i,p_{i+1}]$ become collinear if one ``unfolds'' the planes $P_{i-1}$ and $P_i$ at their intersection line.
	\end{enumerate}
	This can be achieved by unfolding all the planes, tracing a straight geodesic from $g(a)$ to $g(b)$, then folding the planes back to their original position.

	For each $i\in\{1,\dots,n\}$, let $\alpha_i\coloneqq\pi-\widehat{p_{i-1},p_i,p_{i+1}}$ be the exterior angle formed by the segments $[p_{i-1},p_i]$ and $[p_{i},p_{i+1}]$ at $p_i$. The second condition above ensures that $\alpha_i\leq\theta(P_{i-1},P_i)$ for each $i\in\{1,\dots,n\}$, hence, by \eqref{eq:proof_local_injectivity_of_slightly_bent_arcs},
	\begin{equation}\label{eq:proof_local_injectivity_of_slightly_bent_arcs_2}
		\sum_{i=1}^n\alpha_i<\pi.
	\end{equation}

	Assume that $g(a)=g(b)$. Then $p$ is a closed polygonal path with exterior angles $\alpha_1,\dots,\alpha_n$ and an additional exterior angle $\alpha_0$ at the point $g(a)=g(b)$. From Fenchel's theorem for non-positively curved spaces (e.g.\ see~\cite[Corollary 2.4]{MR1459103}), it follows that
	\begin{equation*}
		\sum_{i=0}^n \alpha_i\geq 2\pi.
	\end{equation*}
	Thus, since $\alpha_0\leq\pi$,
	\begin{equation*}
		\sum_{i=1}^n\alpha_i\geq 2\pi-\alpha_0\geq\pi,
	\end{equation*}
	contradicting \eqref{eq:proof_local_injectivity_of_slightly_bent_arcs_2}.
\end{proof}

\subsection{Continuity of the bending map under high Dehn filling}\label{sec:continuity_bending_Dehn_filling}
\begin{proposition}\label{prop:continuous_convergence_drilling}
	The maps $\bb_n\colon\cMP(M)\to\cML(\partial M)$ converge continuously to the bending map $\bb\colon\cMP(M)\to\cML(\partial M)$, i.e.\ for any sequence $\sigma_n\to\sigma$ in $\cMP(M)$, we have $\bb_n(\sigma_n)\to\bb(\sigma)$ in $\cML(\partial M)$.
\end{proposition}
\begin{proof}
	Let $\sigma_n\to\sigma$ in $\cMP(M)$. It suffices to show that, up to a subsequence, we have $\bb_{n}(\sigma_n)\to\bb(\sigma)$. Indeed, a sequence converges to a limit if and only if every subsequence has a further subsequence converging to that limit.

	By Corollary~\ref{cor:convergence_dehn_filling}, the Dehn fillings $\sigma'_n\coloneqq F_n(\sigma_n)$ converge strongly to $\sigma$. Thus, there are nearly isometric embeddings
	\begin{equation*}
		\phi_n\colon (M,\sigma)\dasharrow (M_n,\sigma'_n),
	\end{equation*}
	defined on larger and larger compact subsets of $(M,\sigma)$. For $\epsilon>0$ and for all $n$ large enough, the domain of $\phi_n$ contains the $\epsilon$-truncated convex core ${C(\sigma)}^{\geq\epsilon}$, i.e.\ its intersection with the $\epsilon$-thick part of $\sigma$. By a theorem of McMullen~\cite[Theorem 4.1]{MR1726737}, strong convergence implies that the sequence of compact subsets $(\phi_n^{-1}(C(\sigma'_n)^{\geq\epsilon}))_{n\in\N}$ of $(M,\sigma)$ converges to $C(\sigma)^{\geq\epsilon}$ with respect to the Hausdorff distance.

	From now on, the argument is very close to the argument in \cite[\S 4]{MR2464096} and proceeds in three steps:
	\begin{enumerate}
		\item We first show that the convex-core boundaries converge to a limiting convex pleated surface (up to a subsequence);
		\item We identify this limiting pleated surface with the convex-core boundary of $\sigma$;
		\item We conclude convergence of the bending laminations by Proposition~\ref{prop:Lecuire_continuity_bending_lamination_if_pleated_surfaces_converge} (see~\cite{MR2464096}).
	\end{enumerate}

	Choose $\epsilon>0$ small enough so that the $\epsilon$-thin part of $\sigma$ consists only of its rank-two cusps and such that they lie deep in $C(\sigma)$, and such that the $\epsilon$-thin part of $\sigma'_n$ consists only of the filling tubes.

	Let $S\subseteq\partial M$ be a connected component, corresponding to a connected component of $\partial M_n$ as well, for each $n\in\N$, via the inclusions $M\monic M_n$. Let $\partial_S C(\sigma)$ and $\partial_S C(\sigma'_n)$, $n\in\N$, be the corresponding components of the boundary of the convex cores. By Hausdorff convergence and compactness, there is a sequence of baseframes $\tau_n$ on $\partial_S C(\sigma'_n)$ such that $\phi_n^{-1}(\tau_n)$ converges to a baseframe $\tau$ on $\partial_S {C(\sigma)}^{\geq\epsilon}$. We may choose those baseframes to be based at a point in the interior of a plaque of the boundary of the convex core, such that the first two vectors of $\tau$ provide a baseframe of that plaque, while the third vector is normal to the plaque and pointing out of the convex core.

	Pick a fixed baseframe $\tau^3$ in $\HH^3$ and $\tau^2$ in $\HH^2$. Consider the framed universal coverings
	\begin{equation*}
		\pi_{n}\colon(\HH^3,\tau^3)\to (M_n,\tau_n)\text{ and }\pi\colon (\HH^3,\tau^3)\to (M,\tau),
	\end{equation*}
	whose induced holonomy representations $\rho_n$ (precomposed with $\pi_1(M)\epic\pi_1(M_n)$) converge strongly to $\rho$.

	Let $\widetilde{f}_n\colon\HH^2\to\HH^3$ and $\widetilde{f}\colon\HH^2\to\HH^3$ be the lifted pleated maps, sending $\tau^2$ to $\tau^3$. The induced framed coverings $\pi_n\circ\widetilde{f}_n\colon\HH^2\to\partial_S C(\sigma'_n)$ and $\pi\circ\widetilde{f}\colon\HH^2\to\partial_S C(\sigma)$, pulled back along $\iota_n\colon S\to\partial_S C(\sigma'_n)\subset M_n$ and $\iota\colon S\to\partial_S C(\sigma)\subset M$, yield discrete and faithful representations $r_n,r\colon\pi_1(S)\to\Isom^+(\HH^2)$. Finally, consider the restricted holonomies $\eta_n\coloneqq\rho_n\circ \iota_*$ and $\eta\coloneqq\rho\circ\iota_*$. This yields pleated surfaces $((\widetilde{f}_n,r_n,\eta_n))_{n\in\N}$ and $(\widetilde{f},r,\eta)$.

	Passing to a subsequence, we can assume that $((\widetilde{f}_n,r_n,\eta_n))_{n\in\N}$ converges to some pleated surface $(\widetilde{g},r',\eta')$. This follows from an argument similar to \cite[\S 5.2]{MR0903850}. More precisely:
	\begin{itemize}
		\item We already know that $\eta_n\to\eta$, since $\rho_n\to\rho$ strongly. Thus, we have $\eta'=\eta$.
		\item The induced metrics on $\partial_S C(\sigma'_n)$ are contained in a compact subset of Teichm\"uller space, since the lengths of curves on them are bounded by a function of the lengths of curves with respect to $m_n$ (using~\cite{MR1810370}) and $m_n\to m$. Thus $(r_n)_{n\in\N}$ has a convergent subsequence.
		\item For each $n\in\N$, $\widetilde{f}_n$ sends $\tau^2$ to a fixed baseframe $\tau^3$, hence Arzel\`a--Ascoli's theorem implies that, up to a subsequence, $(\widetilde{f}_n)_{n\in\N}$ converges compactly to a framed map $\widetilde{g}\colon\HH^2\to\HH^3$, which must be pleated and $\pi_1(S)$-equivariant~\cite[5.2.5]{MR0903850}.
	\end{itemize}
	It remains to argue that $(\widetilde{g},r',\eta)$ actually coincides with $(\widetilde{f},r,\eta)$. Again, the argument is similar to \cite[\S 4]{MR2464096}.

	First, $\widetilde{g}(\HH^2)$ is contained in the convex hull $\CH(\rho)$ since $(\phi_n^{-1}(C(\sigma'_n)^{\geq\epsilon}))_{n\in\N}$ converges to $C(\sigma)^{\geq\epsilon}$ for the Hausdorff distance, and because the $\epsilon$-thick part of the convex core contains its boundary.

	By Lemma~\ref{lem:convex_or_even_pleated_surfaces_closed}, the pleated surface $(\widetilde{g},r',\eta)$ is convex or even, and the sign in the definition must be ``$-$'' because of the choice of framing.
	
	Let $\widetilde{L}\subset\HH^2$ be the pleating locus of $\widetilde{g}$. For any plaque $P\subset\HH^2-\widetilde{L}$, we can find plaques $P_n$ of $(\widetilde{f}_n,r_n,\eta_n)$, $n\in\N$, such that $(H_{P_n}^-)_{n\in\N}$ converges to $H_{P}^-$. Since the convex hull $\CH(\rho_n)$ is contained in $H_{P_n}^-$ for each $n\in\N$, we have $\CH(\rho)\subset H_{P}^-$ at the limit. Thus, $\CH(\rho)$ is contained in $C_{\widetilde{g}}$ (see Section~\ref{sec:pleated_surfaces}). Since $\CH(\rho)$ contains rank-two cusps, its interior cannot be empty, so neither is the interior of $C_{\widetilde{g}}$. Thus, $(\widetilde{g},r',\eta)$ is a convex pleated surface and $\widetilde{g}(\HH^2)$ is contained in $\partial\CH(\rho)$. Since $\widetilde{g}(\tau^2)=\tau^3$, $\widetilde{g}(\HH^2)$ is contained in the boundary component $C$ of $\partial\CH(\rho)$ containing $\tau^3$ and $\pi\restr{C}\colon (C,\tau^3)\to(\partial_S C(\sigma),\tau)$ is a framed covering.

	We claim that $\widetilde{g}\colon\HH^2\to C$ is a covering map. Since $\HH^2$ is complete, $C$ is complete and connected, it suffices to show that $\widetilde{g}\colon\HH^2\to C$ is a local isometry. Since it is a path-isometry, it suffices to show that it is a local homeomorphism or, by invariance of domain, that it is locally injective.

	Let $\widetilde{\mu}$ be the bending measure of $\widetilde{g}$. It has no atomic leaf of weight $\pi$, for otherwise $\CH(\rho)$ would have empty interior. Let $\theta\in (0,\pi)$ be such that all atomic leaves of $\widetilde{\mu}$ have weight $<\theta$. Thus, there exists $\delta>0$ such that any geodesic arc $k\subset\HH^2$ of length $\leq\delta$ satisfies $i(\widetilde{\mu},k)<\pi$. By Lemma~\ref{lem:convex_pleated_surface_locally_injective}, the restriction of $\widetilde{g}$ to any $\delta/2$-ball in $\HH^2$ is injective.
	
	Therefore, $\widetilde{g}$ is a covering map. Since $\HH^2$ is simply connected, $\widetilde{g}$ is a universal covering. By framing, it coincides with $\widetilde{f}$.

	Therefore, $(\widetilde{g},r',\eta)=(\widetilde{f},r,\eta)$ and Proposition~\ref{prop:Lecuire_continuity_bending_lamination_if_pleated_surfaces_converge} implies that $\bb_n(\sigma_n)\to\bb(\sigma)$, as required.
\end{proof}

\subsection{Injectivity of the bending map}\label{sec:conclusion_proof_minimally_parabolic}
We can now prove the main theorem of this section.
\begin{proof}[Proof of Theorem~\ref{thm:minimally_parabolic}]
	We need to prove that the bending map $\bb\colon\cMP(M)\to\cML(\partial M)$ is a homeomorphism onto its image. Since $M$ is assumed to have at least one rank-two cusp, it admits no Fuchsian structure. By Lecuire's Theorem~\ref{thm:Lecuire_continuous_proper_on_non_Fuchsian_locus}, $\bb$ is continuous and proper onto its image~\cite{MR2464096,lecuire2025propernessbendingmap}, in particular it has compact fibres. It remains to show that $\bb$ is injective.

	By Proposition~\ref{prop:continuous_convergence_drilling}, the sequence $(\bb_n)_{n\in\N}$ converges continuously to $\bb$. In order to apply the variant of Finney's Theorem (Corollary~\ref{cor:finney_variant}), we need to check that the first condition is satisfied.

	Let $\lambda\in\bb(\cMP(M))\subset\cML(\partial M)$ be a realizable bending lamination. Let $U\subset\cMP(M)$ be a relatively compact neighborhood of $K\coloneqq \bb^{-1}(\lambda)$. We claim that, for all $n$ large enough and for all $\sigma\in U$, the Dehn filling $F_n(\sigma)\in\cCC(M_n)$ is non-Fuchsian. Assume not, then passing to a subsequence, there exists a sequence $(\sigma_n)_{n\in\N}$ in $U$ such that $F_n(\sigma_n)\in\cCC(M_n)$ is Fuchsian for all $n\in\N$. By relative compactness of $U$, we can assume that $\sigma_n\to\sigma'\in\overline{U}$. Then $F_n(\sigma_n)\to\sigma'$ strongly by Corollary~\ref{cor:convergence_dehn_filling}, which is impossible since each $F_n(\sigma_n)$ has a convex core of volume zero while $\sigma'$ has a convex core of strictly positive volume.

	By the above paragraph, for $n$ large enough, $F_n(U)\subset\cCC^{\nf}(M_n)$ and $\bb_n\restr{U}=\bb\circ F_n\restr{U}$ is injective by the convex co-compact bending parameterization Theorem~\ref{thm:injectivity_convex_cocompact} from \cite{dularschlenker2024pleating}.

	Note also that $\cMP(M)\cong\cT(\partial M)$ and $\cML(\partial M)$ are manifolds of the same dimension.

	Therefore, by the variant of Finney's theorem (Corollary~\ref{cor:finney_variant}), compact ANR fibres of $\bb$ are contractible. By Proposition~\ref{prop:real_analycity}, Lemma~\ref{lem:real_analytic_spaces_are_ANRs} and the Corollary~\ref{cor:contractible_compact_real_analytic_spaces_are_singletons} of Borel--Haefliger Theorem~\ref{thm:real_analytic_spaces_have_fundamental_class}, fibres of $\bb$ are singletons, i.e.\ $\bb$ is injective.
\end{proof}

\section{Geometrically finite structures}\label{sec:geometrically_finite_structures}
The goal of this section is to prove Theorem~\ref{thm:injectivity_bending_map_rank_one_cusps}, which states that the bending map for non-Fuchsian geometrically finite structures
\begin{equation*}
	\bb\colon\cGF^{\nf}(M)\to\cBL(M)
\end{equation*}
is a homeomorphism for any hyperbolizable 3-manifold $M$. Surjectivity of $\bb$ is known by the realizability criterion (Theorem~\ref{thm:realizability_criterion}) of Bonahon--Otal and Lecuire, continuity and properness by Lecuire's work~\cite{MR2464096,lecuire2025propernessbendingmap}. The strategy is to apply Theorem~\ref{thm:fibres_contractible_boundary_of_homeo} to $\bb$ in order to conclude that its fibres are contractible and then argue as in the convex co-compact case using real-analyticity.

In Sections~\ref{sec:augmented_teichmuller_space} and~\ref{sec:conformal_parameterization_of_geometrically_finite_structures}, we show that $\cGF(M)$ is homeomorphic to a subspace of the augmented Teichm\"uller space of $\partial M$, from which the local product condition will follow. In Section~\ref{sec:space_of_bending_laminations_of_geometrically_finite_structures}, we describe the space $\cBL(M)$ of bending laminations of geometrically finite structures on $M$ and its tubular topology, defined by Lecuire. Finally, in Section~\ref{sec:contractibility_at_infinity_geometrically_finite_structures}, we check the contractibility at infinity condition and conclude the proof in Section~\ref{sec:injectivity_bending_map_rank_one_cusps}. Section~\ref{sec:punctured_surface_groups} is a brief discussion on the case of quasi-Fuchsian structures over punctured surfaces.

\subsection{Augmented Teichm\"uller space}\label{sec:augmented_teichmuller_space}
See Abikoff~\cite[Chapter II, \S 3]{MR590044} and Wolpert~\cite{MR2039996} for more details.
Let $S$ be a surface, $P\subset S$ a simple multicurve and $P\cup Q$ a maximal simple multicurve, i.e.\ a pants decomposition of $S$, together with a choice of \emph{seams} in order to define twist parameters. Let $p$ and $q$ be the number of components of $P$ and $Q$, respectively.

The \emph{Fenchel--Nielsen coordinates} with respect to $P\cup Q$ assign to each hyperbolic structure $\rho\in\cT(S)$ the tuple
\[
	(\ell_{\rho}(c),\tau_{\rho}(c))_{c\subset P\cup Q}
\]
where $\ell_{\rho}(c)\in\R_{>0}$ is the length of the curve $c$ with respect to $\rho$, and $\tau_{\rho}(c)\in\R$ is the twist parameter of $\rho$ along $c$. It will be more convenient to write $(\ell_{\rho}(c),\tau_{\rho}(c))$ as the complex number $\tau_{\rho}(c)+i(\ell_{\rho}(c))^{-1}\in\HH^2$. The Fenchel--Nielsen coordinates provide a homeomorphism
\begin{equation*}
	\cT(S)\to (\HH^2)^{p+q}\colon \rho\mapsto {(\tau_{\rho}(c)+i(\ell_{\rho}(c))^{-1})}_{c\subset P\cup Q}.
\end{equation*}

By allowing the lengths of the components of $P$ to decrease to zero, we obtain the \emph{extended Fenchel--Nielsen coordinates} for the union
\begin{equation*}
	\bigcup_{\emptyset\subseteq P'\subseteq P}\cT(S;P')
\end{equation*}
of the strata of the augmented Teichm\"uller space $\augmTeich(S)$, where $\cT(S;P')$ denotes the stratum with the simple multicurve $P'\subset S$ pinched and equals the Teichm\"uller space of the noded surface $S-P'$. Whenever a component of $P$ has length zero, its twist parameter is not defined~: there is no twist along nodes. The length $\ell(c)$ going to zero means that the imaginary part of $\tau(c)+i(\ell(c))^{-1}$ goes to infinity. More precisely, we obtain a bijection
\begin{equation}\label{eq:extended_FN}
	\bigcup_{\emptyset\subseteq P'\subseteq P}\cT(S;P')\to (\HH^2\cup\{\infty\})^p\times (\HH^2)^{q}.
\end{equation}

For the purpose of this paper, we focus on a specific stratum $\cT(S;P)$, which lies on the boundary of the main stratum $\cT(S)=\cT(S;\emptyset)$. We denote by $\cT(S;\emptyset,P)$ the union of those two strata, with the subspace topology induced from the augmented Teichm\"uller space $\overline{\cT}(S)$, described explicitly below. The image of $\cT(S;\emptyset,P)$ under the bijection~\eqref{eq:extended_FN} equals $\widehat{(\HH^2)^p}\times(\HH^2)^q$, where $\widehat{(\HH^2)^p}$ is the subspace $(\HH^2)^p\cup \{\infty\}^p$ of $(\HH^2\cup\{\infty\})^p$.

\begin{definition}\label{def:topology_on_union_of_strata}
	The union $\cT(S;\emptyset,P)=\cT(S)\cup\cT(S;P)$ is endowed with the geometric topology. In particular~:
	\begin{enumerate}
		\item The subspaces $\cT(S)$ and $\cT(S;P)=\cT(S-P)$ both inherit their usual topology~;
		\item The subspace $\cT(S;P)$ is closed in $\cT(S;\emptyset,P)$~;
		\item A sequence of metrics $(m_n)_{n\in\N}\subset\cT(S)$ converges to a point $m_{\infty}\in\cT(S;P)$ if and only if the following conditions hold~:
		\begin{enumerate}
			\item We have $\ell_{m_n}(P)\to 0$ when $n\to\infty$~;
			\item For any essential closed curve $c\subset S-P$, $\ell_{m_n}(c)\to\ell_{m_{\infty}}(c)$ when $n\to\infty$.
		\end{enumerate}
	\end{enumerate}
\end{definition}

Similarly, the topology on $\HH^2\cup\{\infty\}$ is such that a basis of neighborhoods of $\infty$ is given by the family of upper half-plane regions $(\{\infty\}\cup\{z\in\HH^2\mid\Im(z)>y\})_{y>0}$. In other words, a sequence $(z_n)_{n\in\N}$ in $\HH^2\cup\{\infty\}$ converges to $\infty$ if and only if $\Im(z_n)\to +\infty$.

\begin{proposition}\label{prop:product_structure_augmented_T}
	The bijection induced by the extended Fenchel--Nielsen coordinates
	\begin{equation*}
		\cT(S;\emptyset,P)\to \widehat{(\HH^2)^{p}}\times (\HH^2)^q\cong\widehat{(\HH^2)^{p}}\times\cT(S-P)
	\end{equation*}
	is a homeomorphism.
\end{proposition}
\begin{proof}
	First, the usual topology on $\cT(S)$ coincides with both the topology induced from Fenchel--Nielsen coordinates, and with the topology induced from the marked length spectrum~\cite{MR3053012}. Actually, the length of curves not intersecting $P$ does not depend on twist parameters along components of $P$. Conversely, twist parameters along components of $Q$ are determined by lengths of curves disjoint from $P$. The statement then follows. See~\cite[\S 3.3]{MR2742784} for more details, and also~\cite[\S 1.3]{MR442293}.
\end{proof}

\subsection{Conformal parameterization of geometrically finite structures}\label{sec:conformal_parameterization_of_geometrically_finite_structures}
By the quasi-conformal deformation theory of Kleinian groups, the space $\cGF(M)\subset\cSI(M)$ of isotopy classes of geometrically finite structures on $M$, endowed with the strong topology, is in bijection with a union of strata
\begin{equation*}
	\augmTeichPinch(\partial M)\coloneqq\bigcup_{P'}\cT(\partial M;P')\subset\augmTeich(\partial M),
\end{equation*}
where $P'$ runs over the \emph{pinchable simple multicurves} on $\partial M$. This bijection restricts to a homeomorphism on each stratum, and is actually a homeomorphism by the following theorem. In the incompressible boundary case, this is shown by Anderson and Lecuire in~\cite[Theorem D]{MR3134412} when restricting to structures without degenerate ends. For the compressible case, we provide a proof which is a combination and slight variation of Lecuire's proof of properness of the bending map~\cite{lecuire2025propernessbendingmap}, together with the aforementioned incompressible boundary case.

Before stating the result, we need the following lemma. We omit the proof, as it is the same as~\cite[Lemma 5.1]{lecuire2025propernessbendingmap} with fewer steps, using a result of Johannson on the finiteness of the subgroup of $\Mod(M)$ preserving doubly incompressible multicurves, see~\cite[Proposition 27.1]{MR551744} and~\cite[1.1, 2.2, 8.2, VIII]{MR551744} for the required definitions.
\begin{lemma}\label{lem:prop_discontinuous}
	The action of $\Mod(M)$ on $\augmTeichPinch(\partial M)$ is properly discontinuous.
\end{lemma}



\begin{theorem}\label{thm:conformal_boundary_map_homeo}
	The conformal boundary map $\mm\colon\cGF(M)\to\augmTeichPinch(\partial M)$ is a homeomorphism.
\end{theorem}
\begin{proof}
	By the simultaneous uniformization theorem and the quasi-conformal deformation theory of Kleinian manifolds, the map $\mm$ is a bijection and restricts to a homeomorphism on each stratum. It remains to show that it is continuous and proper.

	\textbf{Continuity:} First, the conformal boundary map is continuous by~\cite[Lemma 5.2]{MR3134412}~: strong convergence of the holonomy representations implies Carath\'eodory convergence of the domains of discontinuity~\cite[Theorem 4.6.1]{MR3586015}, which implies compact convergence of their Poincar\'e metrics by Hejhal's multiply-connected version of the Carath\'eodory convergence theorem~\cite[Theorem 1]{MR349989}.
	
	\textbf{Properness:} The proof of properness follows the same outline as Lecuire's proof of properness of the bending map~\cite{lecuire2025propernessbendingmap}. First, we prove properness with respect to algebraic convergence, which is then upgraded into strong convergence, and eventually into convergence of isotopy classes of hyperbolic structures.

	Let $(\sigma_n)_{n\in\N}$ be a sequence in $\cGF(M)$. For each $n\in\N$, $\sigma_n$ has the following associated data~:
	\begin{itemize}
		\item its holonomy representation $\rho_n\in\cSH(M)$ and Kleinian group $\Gamma_n\subset\PSL(2,\C)$~;
		\item an isotopy class of homeomorphism $h_n\colon M\to \HH^3/\Gamma_n\eqqcolon M_n$~;
		\item the induced hyperbolic metric on the conformal boundary $m_n\in\cT(\partial M-P_n)\subset\augmTeichPinch(\partial M)$, where $P_n$ is the parabolic locus of $\sigma_n$~;
		\item its convex core $C(\sigma_n)\subset M_n$~;
		\item the pleated boundary of the convex core $f_n\colon\partial M-P_n\to M_n$, with induced hyperbolic metric $s_n\in\cT(\partial M-P_n)$ and bending lamination $\beta_n\in\cML(\partial M-P_n)$~;
		\item the nearest-point retraction $r_n\colon(\partial_{\infty}\sigma_n,m_n)\to(\partial C(\sigma_n),s_n)$.
	\end{itemize}
	By a theorem of Canary~\cite{MR1810370}, for each closed curve $c\subset\partial_{\infty}\sigma_n$ one has
	\begin{equation}\label{eq:canary_convex_core_boundary_metric}
		\ell_{\sigma_n}(c^*)\leq\ell_{s_n}(r_n(c)^*)< 45\ell_{m_n}(c)e^{\ell_{m_n}(c)/2}
	\end{equation}
	where $r_n(c)^*$ (resp.\ $c^*$) is the geodesic representative in the homotopy class of $r_n(c)$ on $\partial C(\sigma_n)$ (resp.\ in $M_n$). By Sugawa~\cite{MR1837250}, dropping the middle term upgrades the bound into $\ell_{\sigma_n}(c^*)\leq 2\ell_{m_n}(c)e^{\ell_{m_n}(c)/2}$.

	To prove properness, assume that $(m_n)_{n\in\N}$ converges to a metric $m_{\infty}\in\cT(\partial M-P_{\infty})\subset\augmTeichPinch(\partial M)$. By local finiteness of the stratification of the augmented Teichm\"uller space, up to passing to a subsequence (denoted the same), we can assume that the sequence of parabolic loci $(P_n)_{n\in\N}$ is constant, say $P_n=P$ for all $n\in\N$, and that $P\subsetneq P_{\infty}$. Write $\Sigma\coloneqq\partial M-P$ and $\Sigma_{\infty}\coloneqq\partial M-P_{\infty}$.
	
	\textbf{Algebraic convergence:} As in~\cite[Theorem 5.2]{MR4264581}, we can apply Thurston's relative boundedness theorem~\cite[Theorem 3.1]{MR4556468} to deduce that $(\rho_n)_{n\in\N}$ has an algebraically convergent subsequence. More precisely, let $Q\subset\partial M$ be a doubly incompressible multicurve containing $P_{\infty}$. Since $m_n\to m_{\infty}$ and $\ell_{m_{\infty}}(Q)$ is finite, the sequence $(\ell_{m_n}(Q))_{n\in\N}$ is bounded. By~\eqref{eq:canary_convex_core_boundary_metric}, this implies that the sequence $(\ell_{\sigma_n}(Q^*))_{n\in\N}$ is bounded. Therefore, by Thurston's relative boundedness theorem and after appropriate conjugation, we can assume that $(\rho_n)_{n\in\N}$ converges algebraically to a representation $\rho_{\infty}\colon\pi_1(M)\to\PSL(2,\C)$. The next step is to upgrade this algebraic convergence to strong convergence, following Bonahon--Otal~\cite{MR2144972} and Lecuire~\cite{lecuire2025propernessbendingmap}.

	\textbf{Geometric convergence:} Up to passing to a further subsequence, we can also assume that $(s_n)_{n\in\N}$ converges to a metric $s_{\infty}\in\cT(\Sigma_{\infty})$. Indeed, by~\eqref{eq:canary_convex_core_boundary_metric}, we know that $\ell_{s_n}(P_{\infty})\to 0$. Let $C$ be a filling (non-simple) multicurve on $\Sigma_{\infty}$. Since $m_n\to m_{\infty}$, $(\ell_{m_n}(C))_{n\in\N}$ is bounded. By~\eqref{eq:canary_convex_core_boundary_metric} again, $(\ell_{s_n}(r_n(C)^*))_{n\in\N}$ is bounded and the conclusion follows.
	
	\begin{claim}\label{claim:convergence_pleated_surface_AA}
		The sequence of pleated surfaces $(f_n\restr{\Sigma_{\infty}},s_n\restr{\Sigma_{\infty}},\Gamma_n)_{n\in\N}$ has a convergent subsequence, whose limit is the boundary of the convex core of $\HH^3/\rho_{\infty}(\pi_1(M))$. In particular, $\rho_{\infty}$ is geometrically finite.
	\end{claim}
	\begin{proof}[Proof of claim~\ref{claim:convergence_pleated_surface_AA}]
		Let $F$ be a connected component of $\Sigma_{\infty}$. Since $P_{\infty}$ is pinchable, there exists a geometrically finite structure on $M$ having $F$ as an ideal boundary component, facing a boundary component of its convex core. Since the image of $\pi_1(F)\to\Isom^+(\HH^3)$, via the holonomy representation, is non-elementary, there is an essential (not necessarily simple) closed curve $c$ on $F$ such that $\rho_{\infty}(c)$ is loxodromic, and hence $\rho_n(c)$ is loxodromic for all $n$ (up to passing to a subsequence), as in~\cite[Lemma 17]{MR2144972}.
		
		Let $c_n$ and $c_n^*$ denote the geodesic representatives of $c$ on $(F,s_n)$ and inside $M_n$, respectively.

		\begin{claim}\label{claim:distance_pleated_surface_uniformly_bounded}
			The distance $d_n$ between $f_n(c_n)$ and $c_n^*$ inside $M_n$ is bounded uniformly in $n\in\N$.
		\end{claim}
		\begin{proof}[Proof of claim~\ref{claim:distance_pleated_surface_uniformly_bounded}]
			In~\cite[Lemma 17]{MR2144972}, the same statement is proved under the assumption that $f_n(c_n)$ meets only finitely many bending lines. This can be circumvented by replacing $f_n(c_n)$ with a \emph{polygonal approximation} $\cP_n$ and using the \emph{error estimate} of Keen and Series~\cite[Proposition 4.8]{MR1331998}, see Section~\ref{sec:polygonal_approximations}.

			Pick a sequence of positive numbers $(\epsilon_n)_{n\in\N}$ converging to zero (and all $<\log(3)/2$). For each $n\in\N$, let $\cP_n$ be a $(\epsilon_n,s(\epsilon_n))$-approximation to $f_n(c_n)$ (which exists by~\cite[Lemma 3.2]{MR2464096}). By~\cite[Lemma 17]{MR2144972}, the distance $d'_n$ between $\cP_n$ and $c_n^*$ inside $M_n$ satisfies
			\begin{equation*}
				d'_n\leq\frac{\sum_{\cP_n}\theta_i}{\ell_{\sigma_n}(c_n^*)}.
			\end{equation*}
			By the error estimate (Proposition~\ref{prop:error_estimate_keen_series}), the distance $d_n$ between $f_n(c_n)$ and $c_n^*$ then satisfies
			\begin{align*}
				d_n&\leq \epsilon_n+d'_n\leq \epsilon_n+\frac{\sum_{\cP_n}\theta_i}{\ell_{\sigma_n}(c_n^*)}\leq \epsilon_n+\frac{i(\beta_n,f_n(c_n))+K\epsilon_n\ell_{s_n}(c_n)}{\ell_{\sigma_n}(c_n^*)}\\
				&\leq \epsilon_n+(\Bavg+K\epsilon_n)\frac{\ell_{s_n}(c_n)}{\ell_{\sigma_n}(c_n^*)}
			\end{align*}
			where the second line uses Bridgeman's uniform bound $\Bavg$ on the average bending of a convex pleated surface~\cite[Theorem 1]{MR1621436}. The right-hand side converges to $\Bavg\ell_{s_{\infty}}(c)/\ell_{\rho_{\infty}}(c)$ (and $\ell_{\rho_{\infty}}(c)>0$ by assumption) and is therefore bounded uniformly in $n\in\N$.
		\end{proof}

		Pick a basepoint $O\in\HH^3$. Since the axes of $\rho_n(c)$ converge to the axis of $\rho_{\infty}(c)$, we can assume that $O$ lies on the axis of $\rho_n(c)$ for all $n\in\N$ (up to conjugating by isometries tending to the identity).

		For each $n\in\N$, pick $z_n\in c_n$ and lift $f_n\restr{F}$ to a convex pleated surface $g_n\coloneqq\widetilde{f_{n}\restr{F}}\colon\widetilde{F}\to\HH^3$, where $\widetilde{F}$ is the cover of $F$ corresponding to the kernel of $\pi_1(F)\to\pi_1(M)$, with basepoint $\widetilde{z_n}$ lifting $z_n$. By Claim~\ref{claim:distance_pleated_surface_uniformly_bounded}, we can assume that $g_n(\widetilde{z_n})$ is at uniformly bounded distance from $O$. Since $c$ is a closed curve on $F$ (and up to a subsequence), we can assume that $z_n\to z_{\infty}\in c\subset F$ and $\widetilde{z_n}\to\widetilde{z_{\infty}}$.

		As in~\cite[\S 5]{MR0903850}, we want to apply Arzel\`a--Ascoli's theorem to the sequence of pleated surfaces $(g_n)_{n\in\N}$. We need to check~:
		\begin{itemize}
			\item \emph{Equicontinuity}~: Since each $g_n$ is a pleated surface, it is $1$-Lipschitz with respect to the induced metric $\widetilde{s_n\restr{F}}$ on $\widetilde{F}$ and the hyperbolic metric on $\HH^3$, and $(\widetilde{s_n\restr{F}})_{n\in\N}$ converges uniformly on compacts to $\widetilde{s_{\infty}\restr{F}}$, which concludes.
			\item \emph{Pointwise relative compactness}~: let $x\in\widetilde{F}$. Then
			\[
				d_{\HH^3}(g_n(x),O)\leq d_{\widetilde{s_n\restr{F}}}(x,\widetilde{z_{n}})+d_{\HH^3}(g_n(\widetilde{z_n}),O),
			\]
			which is uniformly bounded, since $\widetilde{z_n}\to \widetilde{z_{\infty}}$, $\widetilde{s_n\restr{F}}\to\widetilde{s_{\infty}\restr{F}}$ uniformly on a compact subset of $\widetilde{F}$ large enough to contain geodesic segments between each $\widetilde{z_n}$ and $x$, and $g_n(\widetilde{z_n})$ is at uniformly bounded distance of $O$.
		\end{itemize}
		Therefore, up to a subsequence, $(g_n)_{n\in\N}$ converges to a continuous map $g\colon\widetilde{F}\to\HH^3$ uniformly on compacts. Consequently, $g$ is equivariant with respect to the deck transformation group of $\widetilde{F}\to F$ and $\rho_{\infty}$, $1$-Lipschitz with respect to $\widetilde{s_{\infty}\restr{F}}$ and locally convex. Up to a subsequence, the bending loci $(\abs{\beta_n\restr{F}})_{n\in\N}$ converge to a geodesic lamination $L$ on $F$. If $\widetilde{L}$ denotes its preimage in $\widetilde{F}$, then $g$ maps isometrically components of $\widetilde{L}$ to geodesics in $\HH^3$ and components of $\widetilde{F}-\widetilde{L}$ to totally geodesic planes in $\HH^3$. Thus, $g$ is a pleated surface in $\HH^3$ and descends to a convex or even pleated surface in $\HH^3/\rho_{\infty}(\pi_1(M))$~\cite[Lemma 3.1]{MR2464096}. It can only be even if $\rho_{\infty}$ is Fuchsian (see the proof of \cite[Claim 4.6]{lecuire2025propernessbendingmap}), in which case it must be geometrically finite. Otherwise, proceeding in the same way for each component $F$ of $\Sigma_{\infty}$, one obtains a convex pleated surface in $\HH^3/\rho_{\infty}(\pi_1(M))$ which bounds its convex core and $\rho_{\infty}$ is then geometrically finite, following Lecuire's argument in~\cite[Lemma 4.7]{lecuire2025propernessbendingmap}.
	\end{proof}

	Knowing that $\rho_{\infty}$ is geometrically finite and controlling the length of accidental parabolics along the sequence, algebraic convergence of $(\rho_n)_{n\in\N}$ to $\rho_{\infty}$ implies strong convergence as in~\cite[Lemma 4.8]{lecuire2025propernessbendingmap}. This follows from Kleineidam's criterion~\cite[Theorem 1]{MR2153905}. Alternatively, one can use McMullen's strong convergence criterion~\cite[Theorem 1.4]{MR1726737} combined with Sugawa's bound~\cite[Proposition 6.1]{MR1837250}, which ensures that accidental parabolics converge \emph{horocyclically}.

	\textbf{Isotopy markings:} It remains to upgrade the strong convergence of holonomy representations $(\rho_n)_{n\in\N}$ to strong convergence of the isotopy classes of metrics $(\sigma_n)_{n\in\N}$. This follows from proper discontinuity of the action of $\Mod(M)$ on $\augmTeichPinch(\partial M)$ (Lemma~\ref{lem:prop_discontinuous}) and continuity of $\mm$, by the same argument as in~\cite[Theorem 1.2]{lecuire2025propernessbendingmap}.
\end{proof}

We focus on the union $\cGF(M;\emptyset,P)=\cGF(M;\emptyset)\cup\cGF(M;P)$.
\begin{corollary}\label{cor:conformal_boundary_P}
	The conformal boundary map $\mm_P\colon \cGF(M;\emptyset,P)\to\cT(\partial M;\emptyset,P)$ is a homeomorphism. Composing it with the homeomorphism from Proposition~\ref{prop:product_structure_augmented_T} gives the homeomorphism
	\[\cGF(M;\emptyset,P)\rightarrow \widehat{(\HH^2)^p}\times \cT(\partial M-P).\]
\end{corollary}

\subsection{Bending laminations and the tubular topology}\label{sec:space_of_bending_laminations_of_geometrically_finite_structures}
By the realizability criterion of Bonahon--Otal and Lecuire (see Theorem~\ref{thm:realizability_criterion}), the image of $\cGF^{\nf}(M)$ under the bending map is the subset $\cP(M)\subset\cML(\partial M)$, consisting of measured laminations without closed leaves of weight $>\pi$ and satisfying the annulus and disk conditions. Also the image of $\cGF^{\nf}(M;\emptyset)$ is the subset $\cP_{\nc}(M)\subset\cP(M)$ of measured laminations without closed leaves of weight $\geq\pi$.

Following Lecuire, we now describe the topology that needs to be put on $\cP(M)$ to make the bending map continuous, as mentioned in Section~\ref{sec:bending_lamination}.

Let us write $S\coloneqq\partial M$ to lighten notation. For $\lambda\in\cML(S)$, let $\lambda^{(\pi)}$ denote the sublamination of $\lambda$ consisting of its closed leaves of weight $\geq\pi$. Recall that $\abs{\lambda^{(\pi)}}$ denotes the support of $\lambda^{(\pi)}$, which is a simple multicurve. Let also $\cML_{\leq\pi}(S)$ (resp.\ $\cML_{<\pi}(S)$) denote the subset of $\cML(S)$ consisting of measured laminations whose closed leaves all have weight $\leq\pi$ (resp.\ $<\pi$).
\begin{definition}
	Two measured laminations $\mu$ and $\lambda$ are equivalent, denoted $\mu\sim\lambda$, if $\abs{\mu^{(\pi)}}=\abs{\lambda^{(\pi)}}$ and $\mu-\mu^{(\pi)}=\lambda-\lambda^{(\pi)}$.
\end{definition}
In other words, $\mu\sim\lambda$ if and only if they differ only by the weight of their closed leaves of weight $\geq\pi$. Equivalence classes of points of $\cML_{<\pi}(S)$ are singletons. Let $p_{\sim}\colon\cML(S)\to\cML(S)/\sim$ denote the quotient map, sending an element $\lambda$ to its class $[\lambda]$. Note that $p_{\sim}\restr{\cML_{\leq\pi}(S)}$ is a bijection. Define the map $c\colon \cML(S)/\sim\;\to\cML(S)$ sending an equivalence class $[\lambda]$ to its unique representative in $\cML_{\leq\pi}(S)$.

Before defining the topology on $\cML(S)/\sim$, let us recall a way to define the usual (weak-$\ast$) topology on $\cML(S)$. To define a local basis of neighborhoods of a point $\lambda\in\cML(S)$, one first picks a large enough finite collection of generic arcs $\{k_i\}_{i=1}^N$ transverse to $\lambda$. For each connected component of $\lambda$ which is a simple closed curve, pick a generic arc $k_i$ intersecting it once, and for each other component (the exceptional minimal ones) pick a collection of arcs given by the sides of rectangles that form a train track carrying that component. Next, complete this into a finite collection of arcs $\{k_i\}_i$ so that any geodesic on $\partial M$ intersects at least one arc. Then the local basis at $\lambda$ is given by the open sets of the form
\begin{equation}
	B_{\epsilon}(\lambda;\{k_i\}_{i})\coloneqq\left\{\mu\in\cML(S)\mid \abs{i(\lambda,k_i)-i(\mu,k_i)}<\epsilon\text{ for each }i\right\},
\end{equation}
where $\epsilon>0$.

\begin{definition}[Lecuire~\cite{lecuire2025propernessbendingmap}]\label{def:tubular_topology}
	The \emph{tubular topology} on $\cML(S)/\sim$ is the topology defined as follows~:
	\begin{enumerate}
		\item It restricts to the usual topology on $\cML_{<\pi}(S)$~;
		\item A basis of neighborhoods of a point $[\lambda]$ in $\cML(S)/\sim$ with $\lambda^{(\pi)}\neq 0$ is given by the images under $p_{\sim}$ of the following subsets of $\cML(S)$~: for $\epsilon>0$, let
		\begin{equation*}
			B_{\epsilon}([\lambda])\coloneqq\{\mu\in\cML(S)\mid i(\mu,k_i)>\pi-\epsilon\text{ for }i\leq p\text{ and }\abs{i(\lambda,k_i)-i(\mu,k_i)}<\epsilon\text{ for }i>p\},
		\end{equation*}
		where $\{k_i\}_{i=1}^N$ is a collection of transverse arcs as above, and $k_1,\dots,k_p$ ($p<N$) are the ones intersecting the components of $\lambda^{(\pi)}$.
	\end{enumerate}
	From now on, $\cML(S)/\sim$ is always endowed with the tubular topology.
\end{definition}

By definition of its local bases of neighborhoods, the space $\cML(S)/\sim$ with the tubular topology is first-countable. In~\cite{lecuire2025propernessbendingmap}, it is also shown to be a \emph{regular} space, and that compactness and sequential compactness in it are equivalent.

\begin{definition}
	The \emph{space of bending laminations} of $M$, denoted $\cBL(M)$, is defined as the subspace $\cP(M)/\sim$ of $\cML(\partial M)/\sim$, endowed with the tubular topology. By construction, the bending map postcomposed with the quotient map $p_{\sim}$ gives a map	$\bb\colon\cGF^{\nf}(M)\to\cBL(M)$, called the \emph{bending map} as well.
\end{definition}

With the tubular topology, Lecuire proved that the bending map is continuous~\cite{MR2464096} and proper~\cite{lecuire2025propernessbendingmap}.
\begin{theorem}[Lecuire]\label{thm:properness_bending_map}
	The bending map $\bb\colon\cGF^{\nf}(M)\to\cBL(M)$ is continuous and proper.
\end{theorem}

For our purposes, we do not need to consider the entire space $\cBL(M)$ but only some of its strata. More precisely, we focus on the subspace corresponding to bending laminations of structures in $\cGF^{\nf}(M;\emptyset,P)$, i.e.\ only components of $P$ are allowed to have weight $\pi$, and only simultaneously.

\begin{definition}
	For $P$ a simple multicurve on $\partial M$, let $\cBL(M;P)$ (resp.\ $\cBL(M;\emptyset,P)$) denote the subspace of $\cBL(M)$ obtained as the image of $\cGF^{\nf}(M;P)$ (resp.\ $\cGF^{\nf}(M;\emptyset,P)$) under the bending map.
\end{definition}

\subsection{Contractibility at infinity}\label{sec:contractibility_at_infinity_geometrically_finite_structures}
In order to apply Theorem~\ref{thm:fibres_contractible_boundary_of_homeo}, it remains to check the ``contractible image at infinity'' condition. This requires the following lemma, which collects consequences of Lecuire's lemmas.

\begin{lemma}\label{lem:openness_of_cP}
	Let $\lambda\in\cP(M)$ and suppose that $(\lambda_n)_{n\in\N}$ is a sequence in $\cML(\partial M)$ such that $[\lambda_n]\to [\lambda]$ in $\cML(\partial M)/\sim$, for the tubular topology.

	Then, for all $n$ large enough,
	\begin{enumerate}
		\item $|\lambda_n^{(\pi)}|\subseteq|\lambda^{(\pi)}|$~;
		\item $\lambda_n$ satisfies the \emph{annulus} condition, i.e.\ there exists $\eta>0$ such that for any essential annulus $(A,\partial A)$ in $M$, $i(\partial A,\lambda_n)\geq\eta$~;
		\item $\lambda_n$ satisfies the \emph{disk} condition, i.e.\ for any essential disk $(D,\partial D)$ in $M$, $i(\partial D,\lambda_n)>2\pi$.
	\end{enumerate}
	The same holds for $c([\lambda_n])$.
\end{lemma}
\begin{proof}
	The first point is~\cite[Claim 2.2]{lecuire2025propernessbendingmap}.

	Assume that (2) does not hold. Then, up to a subsequence, there exists a sequence of essential annuli $(A_n,\partial A_n)$ in $M$ such that $i(\partial A_n,\lambda_n)\to 0$. Up to a subsequence, $(\partial A_n)_{n\in\N}$ converges to a geodesic lamination $\alpha$ in the Hausdorff topology. The argument of~\cite[Lemma 4.2]{MR2207784} yields
	\begin{equation*}
		i(\alpha,\lambda)\leq\liminf_{n\to\infty}i(\partial A_n,\lambda_n)=0.
	\end{equation*}
	By~\cite[C1,C4]{MR2207784}, this implies that $\lambda$ does not satisfy the annulus condition and contradicts the fact that $\lambda\in\cP(M)$.

	Similarly, if (3) does not hold, then, up to a subsequence, there is a sequence of essential disks $(D_n,\partial D_n)_{n\in\N}$ such that $(i(\partial D_n,\lambda_n))_{n\in\N}$ converges to some value $\leq 2\pi$ and such that $(\partial D_n)_{n\in\N}$ converges in the Hausdorff topology to a geodesic lamination $L$ containing a \emph{homoclinic leaf} $\ell$, by Casson's criterion~\cite[Theorem B.1]{MR2207784}. By~\cite[Lemma 3.6]{MR2207784}, we must have $i(\ell,\lambda)>2\pi$. As above, this yields
	\begin{equation*}
		i(\ell,\lambda)\leq \liminf_{n\to\infty}i(\partial D_n,\lambda_n)\leq 2\pi,
	\end{equation*}
	a contradiction.

	The last sentence is immediate for (1) and (2), and follows from the proof of~\cite[Lemma 3.5]{MR2258744} for (3).
\end{proof}

\begin{proposition}\label{prop:contractibility_at_infinity_geometrically_finite}
	Any $\lambda\in\cBL(M)$ has a basis of neighborhoods in $\cBL(M)$ whose intersection with $\cBL(M;\emptyset)$ is contractible.

	In particular, for any pinchable multicurve $P\subset\partial M$, the bending map $\cGF^{\nf}(M;\emptyset,P)\to\cBL(M)$ has contractible image at infinity, i.e.\ at points of the stratum $\cBL(M;P)$.
\end{proposition}
\begin{proof}
	Let $S\coloneqq\partial M$. We shall use Dehn--Thurston coordinates on $\cML(S)$ with respect to a (marked) pants decomposition $P'$ containing $P$, see~\cite{MR1144770}, which provide a homeomorphism
	\begin{align*}
		\DT_{P'}\colon\cML(S)&\xlongrightarrow{\cong}(\R^2/\{\pm 1\})^{|P'|}\\
		\mu&\longmapsto \left(\widetilde{m}_{c}(\mu),\abs{t}_{c}(\mu)\right)_{c\in P'}
	\end{align*}
	where $\widetilde{m}_{c}(\mu)\coloneqq \pm i(c,\mu)$ is the \emph{signed intersection number} and $\abs{t}_{c}(\mu)$ is the \emph{absolute twist number} of $\mu$ around $c$, with $\pm$ being the sign of the usual twist number $t_c(\mu)$ of $\mu$ around $c$ (see~\cite[\S 2.6]{MR1144770}, where the \emph{absolute} intersection and \emph{signed} twist are used instead).

	Let $\lambda\in\cBL(M;P)$, i.e.\ $\lambda\in\cP(M)$ and $\abs{\lambda^{(\pi)}}=P$.

	Using Dehn--Thurston coordinates, a basis of tubular neighborhoods $U_{\epsilon}([\lambda])$ of $[\lambda]$ in $\cML(\partial M)$ corresponds to neighborhoods in $(\R^2/\{\pm 1\})^{|P'|}$ of the point $\DT_{P'}(\lambda)=(0,\pi)^{\abs{P}}\times y$ of the form
	\begin{equation*}
		B^{\mathrm{tub}}_{\epsilon}(\DT_{P'}(\lambda))=\left((-\epsilon,\epsilon)\times (\pi-\epsilon,+\infty)\right)^{\abs{P}}\times V_{\epsilon}(y),
	\end{equation*}
	where $V_{\epsilon}(y)$ is an $\epsilon$-neighborhood of $y$ in $(\R^2/\{\pm 1\})^{|P'|-|P|}$.

	By Lemma~\ref{lem:openness_of_cP}, for $\epsilon>0$ small enough,
	\begin{equation*}
		U_{\epsilon}([\lambda])\coloneqq\DT_{P'}^{-1}(B^{\mathrm{tub}}_{\epsilon}(\DT_{P'}(\lambda)))
	\end{equation*}
	only contains measured laminations $\mu$ whose representative $c([\mu])$ belongs to $\cP(M)$ and $|\mu^{(\pi)}|\subseteq P$. Therefore, the intersection $U_{\epsilon}([\lambda])\cap\cBL(M;\emptyset)$ is mapped homeomorphically by $\DT_{P'}$ to
	\begin{equation*}
		\Big(\big((-\epsilon,\epsilon)\times (\pi-\epsilon,+\infty)\big)- \big(\{0\}\times [\pi,+\infty)\big)\Big)^{\abs{P}}\times V_{\epsilon}(y), 
	\end{equation*}
	which is contractible.
\end{proof}


\subsection{Injectivity of the bending map for structures with rank-one cusps}\label{sec:injectivity_bending_map_rank_one_cusps}
We can finally prove that the bending map is injective for geometrically finite structures with rank-one cusps, applying the strategy of~\cite{dularschlenker2024pleating} and Theorem~\ref{thm:fibres_contractible_boundary_of_homeo}. The following is Theorem~\ref{thmB} of the introduction.
\begin{theorem}\label{thm:injectivity_bending_map_rank_one_cusps}
	Let $M$ be a hyperbolizable 3-manifold. The bending map
	\begin{equation*}
		\bb\colon\cGF^{\nf}(M)\to\cBL(M)
	\end{equation*}
	is a homeomorphism.
\end{theorem}
\begin{proof}
	The bending map is surjective by Bonahon--Otal~\cite{MR2144972} and Lecuire~\cite{MR2207784}. It is continuous and proper by Lecuire's Theorem~\ref{thm:properness_bending_map}.

	Let us show that it is injective. Since $\bb$ preserves strata, let us fix a pinchable multicurve $P\subset\partial M$ and consider the bending map
	\begin{equation*}
		\bb_{P}\colon\cGF^{\nf}(M;\emptyset,P)\to\cBL(M;\emptyset,P),
	\end{equation*}
	to which we shall apply Theorem~\ref{thm:fibres_contractible_boundary_of_homeo}.

	The spaces $\cGF^{\nf}(M;\emptyset,P)=\cGF^{\nf}(M;\emptyset)\cup\cGF^{\nf}(M;P)$, and $\cBL(M;\emptyset,P)=\cBL(M;\emptyset)\cup\cBL(M;P)$ are first-countable Hausdorff spaces~: $\cGF(M)\cong\augmTeichPinch(\partial M)$ is metrizable, and $\cBL(M)$ is first-countable by definition of the tubular topology and Hausdorff by~\cite[Claim 2.3]{lecuire2025propernessbendingmap}. The $P$-stratum is also closed and the main stratum is dense in both spaces.

	By the bending parameterization of convex co-compact structures~\cite{dularschlenker2024pleating} or minimally parabolic structures (Theorem~\ref{thm:minimally_parabolic}), the restriction
	\begin{equation*}
		(\bb_{P})_{\circ}\colon\cGF^{\nf}(M;\emptyset)\to\cBL(M;\emptyset)
	\end{equation*}
	is a homeomorphism onto the open subset $\cBL(M;\emptyset)$ of $\cBL(M;\emptyset,P)$.

	By the conformal boundary parameterization Theorem~\ref{thm:conformal_boundary_map_homeo}, and more specifically Corollary~\ref{cor:conformal_boundary_P}, the space $\cGF(M;\emptyset,P)$ is homeomorphic to a product
	\begin{equation*}
		\cGF(M;\emptyset,P)\cong \widehat{(\HH^2)^p}\times \cT(\partial M-P),
	\end{equation*}
	and $\cGF^{\nf}(M;\emptyset,P)$ is open in it, hence it is locally a product near $\cGF^{\nf}(M;P)$.

	By Proposition~\ref{prop:contractibility_at_infinity_geometrically_finite}, the bending map $\bb_{P}$ has contractible image at infinity.

	Fibres of
	\begin{equation*}
		(\bb_{P})_{\infty}\colon\cGF^{\nf}(M;P)\to\cBL(M;P)
	\end{equation*}
	are real-analytic subsets by Proposition~\ref{prop:real_analycity}, hence they are absolute neighborhood retracts (Lemma~\ref{lem:real_analytic_spaces_are_ANRs}), and in particular retracts of relatively compact neighborhoods.

	Therefore, by Theorem~\ref{thm:fibres_contractible_boundary_of_homeo}, fibres of $(\bb_{P})_{\infty}$ are contractible. But they are also compact (by properness) and real-analytic (by Proposition~\ref{prop:real_analycity}), hence they are singletons by Corollary~\ref{cor:contractible_compact_real_analytic_spaces_are_singletons}. This proves that $(\bb_{P})_{\infty}$ is injective.

	It remains to check that $\bb^{-1}$ is continuous or, equivalently, sequentially continuous, by first-countability. Suppose that $\lambda_n\to\lambda$ in $\cBL(M)$. We claim that $\bb^{-1}(\lambda_n)\to\bb^{-1}(\lambda)$ in $\cGF^{\nf}(M)$. It suffices to check that any subsequence of $(\bb^{-1}(\lambda_n))_{n\in\N}$ has a subsequence converging to $\bb^{-1}(\lambda)$, which is true since $\bb$ is continuous, proper and injective.
\end{proof}

\subsection{Punctured surface groups}\label{sec:punctured_surface_groups}
When $S$ is a punctured surface, the space $\cQF(S)$ of quasi-Fuchsian structures on $S\times\R$ appears as a stratum of the space of geometrically finite structures on the interior of $\overline{M}\coloneqq F\times [0,1]$, where $F$ is obtained by removing small open disk neighborhoods of the punctures of $S$. The manifold $\overline{M}$ is a handlebody with boundary the double of $F$. The simple multicurve $P\coloneqq \partial F\times \{1/2\}$ on its boundary is pinchable, and we have
\begin{equation*}
	\cQF(S)=\cGF(\Int(\overline{M});P).
\end{equation*}
Therefore, Theorem~\ref{thm:injectivity_bending_map_rank_one_cusps} implies that the non-Fuchsian subspace $\cQF^{\nf}(S)$ is parameterized by the bending map. This extends the known case of the punctured torus, proved by Series in~\cite{MR2258745} using the pleating invariants of Keen--Series~\cite{MR2052972}.

\bibliographystyle{alpha}
\bibliography{bib.bib}
\end{document}

%% file: commands.tex
\DeclareMathOperator{\Teich}{Teich}
\DeclareMathOperator{\Lip}{Lip}

\usepackage{amsthm}
\newtheorem{theorem}[subsection]{Theorem}
\newtheorem*{theorem*}{Theorem}
\newtheorem{proposition}[subsection]{\rm\bf Proposition}

\newtheorem{lemma}[subsection]{Lemma}
\newtheorem{corollary}[subsection]{Corollary}
\newtheorem{definition}[subsection]{Definition}
\newtheorem{remark}[subsection]{Remark}

\newtheorem{claim}[subsection]{Claim}

\newcommand{\C}{{\mathbb{C}}}
\newcommand{\CP}{{\mathbb{CP}}}
\newcommand{\N}{{\mathbb{N}}}
\newcommand{\bS}{{\mathbb{S}}}

\newcommand{\HH}{{\mathbb{H}}}
\newcommand{\R}{{\mathbb{R}}}

\newcommand{\Z}{{\mathbb{Z}}}

\newcommand{\cCC}{{\mathcal{CC}}}
\newcommand{\cGF}{{\mathcal{GF}}}

\newcommand{\cH}{{\mathcal{H}}}
\newcommand{\cI}{{\mathcal{I}}}

\newcommand{\cSH}{{\mathcal{SH}}}
\newcommand{\cSI}{{\mathcal{SI}}}

\newcommand{\cAH}{{\mathcal{AH}}}
\newcommand{\cQH}{{\mathcal{QH}}}

\newcommand{\cQI}{{\mathcal{QI}}}
\newcommand{\Bavg}{{\mathrm{B}_{\mathrm{avg}}}}
\newcommand{\nf}{\mathrm{nF}}

\newcommand{\cP}{{\mathcal{P}}}

\newcommand{\cT}{{\mathcal{T}}}
\newcommand{\cV}{{\mathcal{V}}}
\newcommand{\cX}{{\mathcal{X}}}

\newcommand{\cML}{{\mathcal{ML}}}
\newcommand{\cBL}{{\mathcal{BL}}}

\newcommand{\cGL}{{\mathcal{GL}}}

\newcommand{\cQF}{{\mathcal{QF}}}
\newcommand{\cMP}{{\mathcal{MP}}}

\newcommand{\Isom}{\rm{Isom}}
\newcommand{\PSL}{\rm{PSL}}
\newcommand{\CH}{\rm{CH}}

\newcommand{\augmTeich}{\overline{\cT}}
\newcommand{\augmTeichPinch}{\overline{\cT}^{\mathrm{pinch}}}

\DeclareMathOperator{\Int}{Int}

\DeclareMathOperator{\DT}{DT}
\DeclareMathOperator{\U}{U}
\DeclareMathOperator{\id}{id}

\DeclareMathOperator{\proj}{proj}
\DeclareMathOperator{\mm}{m}
\DeclareMathOperator{\bb}{b}

\DeclareMathOperator{\Mod}{Mod}

\DeclareMathOperator{\nc}{nc}

\newcommand{\epic}{\twoheadrightarrow}
\newcommand{\monic}{\xhookrightarrow{}}

\newcommand{\restr}[1]{|_{#1}}

\newcommand{\abs}[1]{\left|#1\right|}

%% file: bib.bib
@article {MR2144972,
    AUTHOR = {Bonahon, Francis and Otal, Jean-Pierre},
     TITLE = {Laminations mesur\'ees de plissage des vari\'et\'es
              hyperboliques de dimension 3},
   JOURNAL = {Ann. of Math. (2)},
  FJOURNAL = {Annals of Mathematics. Second Series},
    VOLUME = {160},
      YEAR = {2004},
    NUMBER = {3},
     PAGES = {1013--1055},
      ISSN = {0003-486X,1939-8980},
   MRCLASS = {57M50 (57N10 57R30)},
  MRNUMBER = {2144972},
MRREVIEWER = {Thilo\ Kuessner},
       DOI = {10.4007/annals.2004.160.1013},
       URL = {https://doi.org/10.4007/annals.2004.160.1013},
}

@article{dularschlenker2024pleating,
  title={Convex co-compact hyperbolic manifolds are determined by their pleating lamination},
  author={Dular, Bruno and Schlenker, Jean-Marc},
  journal={to appear in \emph{Annals of Mathematics}},
  note={arXiv:2403.10090},
  year={2024}
}

@incollection{MR4556468,
    AUTHOR = {Thurston, William P.},
     TITLE = {Hyperbolic structures on 3-manifolds, {III}: Deformations of
              3-manifolds with incompressible boundary},
 BOOKTITLE = {Collected works of {W}illiam {P}. {T}hurston with commentary.
              {V}ol. {II}. 3-manifolds, complexity and geometric group
              theory},
     PAGES = {111--129},
      NOTE = {1986 preprint, 1998 eprint},
 PUBLISHER = {Amer. Math. Soc., Providence, RI},
      YEAR = {1986},
      ISBN = {978-1-4704-6389-2; [9781470468347]; [9781470451646]},
   MRCLASS = {57K32},
  MRNUMBER = {4556468},
}

@article {MR3134412,
    AUTHOR = {Anderson, James W. and Lecuire, Cyril},
     TITLE = {Strong convergence of {K}leinian groups: the cracked eggshell},
   JOURNAL = {Comment. Math. Helv.},
  FJOURNAL = {Commentarii Mathematici Helvetici. A Journal of the Swiss
              Mathematical Society},
    VOLUME = {88},
      YEAR = {2013},
    NUMBER = {4},
     PAGES = {813--857},
      ISSN = {0010-2571,1420-8946},
   MRCLASS = {57M50 (30F40)},
  MRNUMBER = {3134412},
MRREVIEWER = {Ken-ichi\ Ohshika},
       DOI = {10.4171/CMH/304},
       URL = {https://doi.org/10.4171/CMH/304},
}

@article {MR2464096,
    AUTHOR = {Lecuire, Cyril},
     TITLE = {Continuity of the bending map},
   JOURNAL = {Ann. Fac. Sci. Toulouse Math. (6)},
  FJOURNAL = {Annales de la Facult\'{e} des Sciences de Toulouse.
              Math\'{e}matiques. S\'{e}rie 6},
    VOLUME = {17},
      YEAR = {2008},
    NUMBER = {1},
     PAGES = {93--119},
      ISSN = {0240-2963,2258-7519},
   MRCLASS = {57M50},
  MRNUMBER = {2464096},
MRREVIEWER = {Thilo\ Kuessner},
       URL = {http://afst.cedram.org/item?id=AFST_2008_6_17_1_93_0},
}

@article {MR1331998,
    AUTHOR = {Keen, Linda and Series, Caroline},
     TITLE = {Continuity of convex hull boundaries},
   JOURNAL = {Pacific J. Math.},
  FJOURNAL = {Pacific Journal of Mathematics},
    VOLUME = {168},
      YEAR = {1995},
    NUMBER = {1},
     PAGES = {183--206},
      ISSN = {0030-8730,1945-5844},
   MRCLASS = {30F40 (57M50)},
  MRNUMBER = {1331998},
MRREVIEWER = {Bruno\ P.\ Zimmermann},
       URL = {http://projecteuclid.org/euclid.pjm/1102620682},
}

@incollection {MR0903852,
    AUTHOR = {Epstein, D. B. A. and Marden, A.},
     TITLE = {Convex hulls in hyperbolic space, a theorem of {S}ullivan, and
              measured pleated surfaces},
 BOOKTITLE = {Analytical and geometric aspects of hyperbolic space
              ({C}oventry/{D}urham, 1984)},
    SERIES = {London Math. Soc. Lecture Note Ser.},
    VOLUME = {111},
     PAGES = {113--253},
 PUBLISHER = {Cambridge Univ. Press, Cambridge},
      YEAR = {1987},
      ISBN = {0-521-33906-5},
   MRCLASS = {52A55 (32G15 57N10 57R25 58F17)},
  MRNUMBER = {903852},
MRREVIEWER = {William\ Dunbar},
}

@incollection {MR0903850,
    AUTHOR = {Canary, R. D. and Epstein, D. B. A. and Green, P.},
     TITLE = {Notes on notes of {T}hurston},
 BOOKTITLE = {Analytical and geometric aspects of hyperbolic space
              ({C}oventry/{D}urham, 1984)},
    SERIES = {London Math. Soc. Lecture Note Ser.},
    VOLUME = {111},
     PAGES = {3--92},
 PUBLISHER = {Cambridge Univ. Press, Cambridge},
      YEAR = {1987},
      ISBN = {0-521-33906-5},
   MRCLASS = {57N10 (32G15 32G99 54A20 57M99 58F17)},
  MRNUMBER = {903850},
MRREVIEWER = {William\ Dunbar},
}

@incollection {MR2039996,
    AUTHOR = {Wolpert, Scott A.},
     TITLE = {Geometry of the {W}eil-{P}etersson completion of
              {T}eichm\"{u}ller space},
 BOOKTITLE = {Surveys in differential geometry, {V}ol. {VIII} ({B}oston,
              {MA}, 2002)},
    SERIES = {Surv. Differ. Geom.},
    VOLUME = {8},
     PAGES = {357--393},
 PUBLISHER = {Int. Press, Somerville, MA},
      YEAR = {2003},
      ISBN = {1-57146-114-0},
   MRCLASS = {32G15 (30F60)},
  MRNUMBER = {2039996},
MRREVIEWER = {Samuel\ Grushevsky},
       DOI = {10.4310/SDG.2003.v8.n1.a13},
       URL = {https://doi.org/10.4310/SDG.2003.v8.n1.a13},
}

@book{MR1144770,
    AUTHOR = {Penner, R. C. and Harer, J. L.},
     TITLE = {Combinatorics of train tracks},
    SERIES = {Annals of Mathematics Studies},
    VOLUME = {125},
 PUBLISHER = {Princeton University Press, Princeton, NJ},
      YEAR = {1992},
     PAGES = {xii+216},
      ISBN = {0-691-08764-4; 0-691-02531-2},
   MRCLASS = {57M99 (30F60 57N05 57R30)},
  MRNUMBER = {1144770},
MRREVIEWER = {Y.\ Minsky},
       DOI = {10.1515/9781400882458},
       URL = {https://doi.org/10.1515/9781400882458},
}

@book{MR0872468,
    AUTHOR = {Daverman, Robert J.},
     TITLE = {Decompositions of manifolds},
    SERIES = {Pure and Applied Mathematics},
    VOLUME = {124},
 PUBLISHER = {Academic Press, Inc., Orlando, FL},
      YEAR = {1986},
     PAGES = {xii+317},
      ISBN = {0-12-204220-4},
   MRCLASS = {57-01 (54B15)},
  MRNUMBER = {872468},
}

@article {MR0224087,
    AUTHOR = {Finney, Ross L.},
     TITLE = {Pseudo-isotopies and cellular sets},
   JOURNAL = {Michigan Math. J.},
  FJOURNAL = {Michigan Mathematical Journal},
    VOLUME = {14},
      YEAR = {1967},
     PAGES = {417--421},
      ISSN = {0026-2285,1945-2365},
   MRCLASS = {55.30},
  MRNUMBER = {224087},
MRREVIEWER = {A.\ H.\ Copeland, Jr.},
       URL = {http://projecteuclid.org/euclid.mmj/1028999842},
}

@article {MR1413855,
    AUTHOR = {Bonahon, Francis},
     TITLE = {Shearing hyperbolic surfaces, bending pleated surfaces and
              {T}hurston's symplectic form},
   JOURNAL = {Ann. Fac. Sci. Toulouse Math. (6)},
  FJOURNAL = {Toulouse. Facult\'e{} des Sciences. Annales. Math\'ematiques.
              S\'erie 6},
    VOLUME = {5},
      YEAR = {1996},
    NUMBER = {2},
     PAGES = {233--297},
      ISSN = {0240-2955},
   MRCLASS = {57M50 (53C15 57N05 57N10)},
  MRNUMBER = {1413855},
MRREVIEWER = {Athanase\ Papadopoulos},
       URL = {http://www.numdam.org/item?id=AFST_1996_6_5_2_233_0},
}

@article {MR4651897,
    AUTHOR = {Baba, Shinpei and Ohshika, Ken'ichi},
     TITLE = {Realisation of bending measured laminations by {K}leinian
              surface groups},
   JOURNAL = {Int. Math. Res. Not. IMRN},
  FJOURNAL = {International Mathematics Research Notices. IMRN},
    VOLUME = {2023},
      YEAR = {2023},
    NUMBER = {19},
     PAGES = {16674--16707},
      ISSN = {1073-7928,1687-0247},
   MRCLASS = {57K20 (30F40)},
  MRNUMBER = {4651897},
MRREVIEWER = {Hongbin\ Sun},
       DOI = {10.1093/imrn/rnac357},
       URL = {https://doi.org/10.1093/imrn/rnac357},
}

@inproceedings {MR278333,
    AUTHOR = {Sullivan, D.},
     TITLE = {Combinatorial invariants of analytic spaces},
 BOOKTITLE = {Proceedings of {L}iverpool {S}ingularities---{S}ymposium, {I}
              (1969/70)},
    SERIES = {Lecture Notes in Math.},
    VOLUME = {Vol. 192},
     PAGES = {165--168},
 PUBLISHER = {Springer, Berlin-New York},
      YEAR = {1971},
   MRCLASS = {57.36 (32.00)},
  MRNUMBER = {278333},
MRREVIEWER = {A.\ H.\ Wallace},
}

@article {MR149503,
    AUTHOR = {Borel, Armand and Haefliger, Andr\'e},
     TITLE = {La classe d'homologie fondamentale d'un espace analytique},
   JOURNAL = {Bull. Soc. Math. France},
  FJOURNAL = {Bulletin de la Soci\'et\'e{} Math\'ematique de France},
    VOLUME = {89},
      YEAR = {1961},
     PAGES = {461--513},
      ISSN = {0037-9484},
   MRCLASS = {57.32 (32.47)},
  MRNUMBER = {149503},
MRREVIEWER = {E.\ Brieskorn},
       URL = {http://www.numdam.org/item?id=BSMF_1961__89__461_0},
}

@incollection {MR1208313,
    AUTHOR = {Kamishima, Yoshinobu and Tan, Ser P.},
     TITLE = {Deformation spaces on geometric structures},
 BOOKTITLE = {Aspects of low-dimensional manifolds},
    SERIES = {Adv. Stud. Pure Math.},
    VOLUME = {20},
     PAGES = {263--299},
 PUBLISHER = {Kinokuniya, Tokyo},
      YEAR = {1992},
      ISBN = {4-314-10077-X},
   MRCLASS = {57M50},
  MRNUMBER = {1208313},
MRREVIEWER = {William\ Goldman},
       DOI = {10.2969/aspm/02010263},
       URL = {https://doi.org/10.2969/aspm/02010263},
}

@article {MR1837250,
    AUTHOR = {Sugawa, Toshiyuki},
     TITLE = {Uniform perfectness of the limit sets of {K}leinian groups},
   JOURNAL = {Trans. Amer. Math. Soc.},
  FJOURNAL = {Transactions of the American Mathematical Society},
    VOLUME = {353},
      YEAR = {2001},
    NUMBER = {9},
     PAGES = {3603--3615},
      ISSN = {0002-9947,1088-6850},
   MRCLASS = {30F40 (30F45)},
  MRNUMBER = {1837250},
MRREVIEWER = {Leonid\ Potyagailo},
       DOI = {10.1090/S0002-9947-01-02775-1},
       URL = {https://doi.org/10.1090/S0002-9947-01-02775-1},
}

@article {MR1810370,
    AUTHOR = {Canary, R. D.},
     TITLE = {The conformal boundary and the boundary of the convex core},
   JOURNAL = {Duke Math. J.},
  FJOURNAL = {Duke Mathematical Journal},
    VOLUME = {106},
      YEAR = {2001},
    NUMBER = {1},
     PAGES = {193--207},
      ISSN = {0012-7094,1547-7398},
   MRCLASS = {57M50 (30F40 30F45)},
  MRNUMBER = {1810370},
MRREVIEWER = {Bruno\ P.\ Zimmermann},
       DOI = {10.1215/S0012-7094-01-10616-9},
       URL = {https://doi.org/10.1215/S0012-7094-01-10616-9},
}

@article {MR1621436,
    AUTHOR = {Bridgeman, Martin},
     TITLE = {Average bending of convex pleated planes in hyperbolic
              three-space},
   JOURNAL = {Invent. Math.},
  FJOURNAL = {Inventiones Mathematicae},
    VOLUME = {132},
      YEAR = {1998},
    NUMBER = {2},
     PAGES = {381--391},
      ISSN = {0020-9910,1432-1297},
   MRCLASS = {57M50 (30F40 30F50 53A35)},
  MRNUMBER = {1621436},
MRREVIEWER = {Michel\ Coornaert},
       DOI = {10.1007/s002220050227},
       URL = {https://doi.org/10.1007/s002220050227},
}

@incollection {MR4264581,
    AUTHOR = {Lecuire, Cyril},
     TITLE = {The double limit theorem and its legacy},
 BOOKTITLE = {In the tradition of {T}hurston---geometry and topology},
     PAGES = {263--290},
 PUBLISHER = {Springer, Cham},
      YEAR = {[2020] \copyright 2020},
      ISBN = {978-3-030-55928-1; 978-3-030-55927-4},
   MRCLASS = {57M30 (30F40 30F60)},
  MRNUMBER = {4264581},
       DOI = {10.1007/978-3-030-55928-1\_8},
       URL = {https://doi.org/10.1007/978-3-030-55928-1_8},
}

@article {MR2207784,
    AUTHOR = {Lecuire, Cyril},
     TITLE = {Plissage des vari\'et\'es hyperboliques de dimension 3},
   JOURNAL = {Invent. Math.},
  FJOURNAL = {Inventiones Mathematicae},
    VOLUME = {164},
      YEAR = {2006},
    NUMBER = {1},
     PAGES = {85--141},
      ISSN = {0020-9910,1432-1297},
   MRCLASS = {57N10 (57M50)},
  MRNUMBER = {2207784},
MRREVIEWER = {Thilo\ Kuessner},
       DOI = {10.1007/s00222-005-0470-z},
       URL = {https://doi.org/10.1007/s00222-005-0470-z},
}

@book {MR3586015,
    AUTHOR = {Marden, Albert},
     TITLE = {Hyperbolic manifolds},
      NOTE = {An introduction in 2 and 3 dimensions},
 PUBLISHER = {Cambridge University Press, Cambridge},
      YEAR = {2016},
     PAGES = {xviii+515},
      ISBN = {978-1-107-11674-0},
   MRCLASS = {57-01 (20H10 30F40 57M50 57N05 57N10)},
  MRNUMBER = {3586015},
MRREVIEWER = {Thilo\ Kuessner},
       DOI = {10.1017/CBO9781316337776},
       URL = {https://doi.org/10.1017/CBO9781316337776},
}

@article {MR2208419,
    AUTHOR = {Schlenker, Jean-Marc},
     TITLE = {Hyperbolic manifolds with convex boundary},
   JOURNAL = {Invent. Math.},
  FJOURNAL = {Inventiones Mathematicae},
    VOLUME = {163},
      YEAR = {2006},
    NUMBER = {1},
     PAGES = {109--169},
      ISSN = {0020-9910,1432-1297},
   MRCLASS = {57M50 (53C21)},
  MRNUMBER = {2208419},
MRREVIEWER = {Igor\ Rivin},
       DOI = {10.1007/s00222-005-0456-x},
       URL = {https://doi.org/10.1007/s00222-005-0456-x},
}

@article {MR1125669,
    AUTHOR = {Labourie, Fran{\c{c}}ois},
     TITLE = {Probl\`eme de {M}inkowski et surfaces \`a{} courbure constante
              dans les vari\'et\'es hyperboliques},
   JOURNAL = {Bull. Soc. Math. France},
  FJOURNAL = {Bulletin de la Soci\'et\'e{} Math\'ematique de France},
    VOLUME = {119},
      YEAR = {1991},
    NUMBER = {3},
     PAGES = {307--325},
      ISSN = {0037-9484,2102-622X},
   MRCLASS = {53C23 (35J60)},
  MRNUMBER = {1125669},
MRREVIEWER = {Philippe\ Delano\"e},
       URL = {http://www.numdam.org/item?id=BSMF_1991__119_3_307_0},
       DOI = {10.24033/bsmf.2169},
 ZBLNUMBER = {0758.53030},
}

@article {MR1163450,
    AUTHOR = {Labourie, Fran{\c{c}}ois},
     TITLE = {M\'etriques prescrites sur le bord des vari\'et\'es
              hyperboliques de dimension {$3$}},
   JOURNAL = {J. Differential Geom.},
  FJOURNAL = {Journal of Differential Geometry},
    VOLUME = {35},
      YEAR = {1992},
    NUMBER = {3},
     PAGES = {609--626},
      ISSN = {0022-040X,1945-743X},
   MRCLASS = {53C20 (53C21 57M50)},
  MRNUMBER = {1163450},
MRREVIEWER = {Viktor\ Schroeder},
       URL = {http://projecteuclid.org/euclid.jdg/1214448258},
       DOI = {10.4310/jdg/1214448258},
 ZBLNUMBER = {0768.53017},
}

@article {MR3001608,
    AUTHOR = {Namazi, Hossein and Souto, Juan},
     TITLE = {Non-realizability and ending laminations: proof of the density
              conjecture},
   JOURNAL = {Acta Math.},
  FJOURNAL = {Acta Mathematica},
    VOLUME = {209},
      YEAR = {2012},
    NUMBER = {2},
     PAGES = {323--395},
      ISSN = {0001-5962,1871-2509},
   MRCLASS = {30F40 (20H10 57M50)},
  MRNUMBER = {3001608},
MRREVIEWER = {Majid\ Heydarpour},
       DOI = {10.1007/s11511-012-0088-0},
       URL = {https://doi.org/10.1007/s11511-012-0088-0},
}

@article {MR2821565,
    AUTHOR = {Ohshika, Ken'ichi},
     TITLE = {Realising end invariants by limits of minimally parabolic,
              geometrically finite groups},
   JOURNAL = {Geom. Topol.},
  FJOURNAL = {Geometry \& Topology},
    VOLUME = {15},
      YEAR = {2011},
    NUMBER = {2},
     PAGES = {827--890},
      ISSN = {1465-3060,1364-0380},
   MRCLASS = {57M50 (30F40)},
  MRNUMBER = {2821565},
MRREVIEWER = {Bruno\ P.\ Zimmermann},
       DOI = {10.2140/gt.2011.15.827},
       URL = {https://doi.org/10.2140/gt.2011.15.827},
}

@article {MR111834,
    AUTHOR = {Bers, Lipman},
     TITLE = {Simultaneous uniformization},
   JOURNAL = {Bull. Amer. Math. Soc.},
  FJOURNAL = {Bulletin of the American Mathematical Society},
    VOLUME = {66},
      YEAR = {1960},
     PAGES = {94--97},
      ISSN = {0002-9904},
   MRCLASS = {30.00},
  MRNUMBER = {111834},
MRREVIEWER = {H.\ L.\ Royden},
       DOI = {10.1090/S0002-9904-1960-10413-2},
       URL = {https://doi.org/10.1090/S0002-9904-1960-10413-2},
}

@book {MR1638795,
    AUTHOR = {Matsuzaki, Katsuhiko and Taniguchi, Masahiko},
     TITLE = {Hyperbolic manifolds and {K}leinian groups},
    SERIES = {Oxford Mathematical Monographs},
      NOTE = {Oxford Science Publications},
 PUBLISHER = {The Clarendon Press, Oxford University Press, New York},
      YEAR = {1998},
     PAGES = {x+253},
      ISBN = {0-19-850062-9},
   MRCLASS = {30F40 (20H10 30F10)},
  MRNUMBER = {1638795},
MRREVIEWER = {I.\ Kra},
}

@article {MR2079598,
    AUTHOR = {Brock, Jeffrey F. and Bromberg, Kenneth W.},
     TITLE = {On the density of geometrically finite {K}leinian groups},
   JOURNAL = {Acta Math.},
  FJOURNAL = {Acta Mathematica},
    VOLUME = {192},
      YEAR = {2004},
    NUMBER = {1},
     PAGES = {33--93},
      ISSN = {0001-5962,1871-2509},
   MRCLASS = {57M50 (30F40)},
  MRNUMBER = {2079598},
MRREVIEWER = {Lee\ Mosher},
       DOI = {10.1007/BF02441085},
       URL = {https://doi.org/10.1007/BF02441085},
}

@article {MR2186972,
    AUTHOR = {Bonahon, Francis},
     TITLE = {Kleinian groups which are almost {F}uchsian},
   JOURNAL = {J. Reine Angew. Math.},
  FJOURNAL = {Journal f\"ur die Reine und Angewandte Mathematik. [Crelle's
              Journal]},
    VOLUME = {587},
      YEAR = {2005},
     PAGES = {1--15},
      ISSN = {0075-4102,1435-5345},
   MRCLASS = {57M50 (30D35 30D40 30F40)},
  MRNUMBER = {2186972},
MRREVIEWER = {Bruno\ P.\ Zimmermann},
       DOI = {10.1515/crll.2005.2005.587.1},
       URL = {https://doi.org/10.1515/crll.2005.2005.587.1},
}

@article {MR2052972,
    AUTHOR = {Keen, Linda and Series, Caroline},
     TITLE = {Pleating invariants for punctured torus groups},
   JOURNAL = {Topology},
  FJOURNAL = {Topology. An International Journal of Mathematics},
    VOLUME = {43},
      YEAR = {2004},
    NUMBER = {2},
     PAGES = {447--491},
      ISSN = {0040-9383},
   MRCLASS = {30F40 (20H10 57M50)},
  MRNUMBER = {2052972},
MRREVIEWER = {John\ R.\ Parker},
       DOI = {10.1016/S0040-9383(03)00052-1},
       URL = {https://doi.org/10.1016/S0040-9383(03)00052-1},
}

@incollection {MR2258745,
    AUTHOR = {Series, Caroline},
     TITLE = {Thurston's bending measure conjecture for once punctured torus
              groups},
 BOOKTITLE = {Spaces of {K}leinian groups},
    SERIES = {London Math. Soc. Lecture Note Ser.},
    VOLUME = {329},
     PAGES = {75--89},
 PUBLISHER = {Cambridge Univ. Press, Cambridge},
      YEAR = {2006},
      ISBN = {978-0-521-61797-0; 0-521-61797-9},
   MRCLASS = {30F40 (20H10 30F35 57M50 57N10)},
  MRNUMBER = {2258745},
MRREVIEWER = {Bruno\ P.\ Zimmermann},
}

@article {MR1049503,
    AUTHOR = {Kra, Irwin},
     TITLE = {Horocyclic coordinates for {R}iemann surfaces and moduli
              spaces. {I}. {T}eichm\"uller and {R}iemann spaces of
              {K}leinian groups},
   JOURNAL = {J. Amer. Math. Soc.},
  FJOURNAL = {Journal of the American Mathematical Society},
    VOLUME = {3},
      YEAR = {1990},
    NUMBER = {3},
     PAGES = {499--578},
      ISSN = {0894-0347,1088-6834},
   MRCLASS = {32G15 (30F35 30F60)},
  MRNUMBER = {1049503},
MRREVIEWER = {Robert\ C.\ Penner},
       DOI = {10.2307/1990927},
       URL = {https://doi.org/10.2307/1990927},
}

@article {MR1726737,
    AUTHOR = {McMullen, Curtis T.},
     TITLE = {Hausdorff dimension and conformal dynamics. {I}. {S}trong
              convergence of {K}leinian groups},
   JOURNAL = {J. Differential Geom.},
  FJOURNAL = {Journal of Differential Geometry},
    VOLUME = {51},
      YEAR = {1999},
    NUMBER = {3},
     PAGES = {471--515},
      ISSN = {0022-040X,1945-743X},
   MRCLASS = {37F30 (30F40 57S30)},
  MRNUMBER = {1726737},
MRREVIEWER = {Edward\ C.\ Taylor},
       URL = {http://projecteuclid.org/euclid.jdg/1214425139},
}

@article {MR2153905,
    AUTHOR = {Kleineidam, G.},
     TITLE = {Strong convergence of {K}leinian groups},
   JOURNAL = {Geom. Funct. Anal.},
  FJOURNAL = {Geometric and Functional Analysis},
    VOLUME = {15},
      YEAR = {2005},
    NUMBER = {2},
     PAGES = {416--452},
      ISSN = {1016-443X,1420-8970},
   MRCLASS = {30F40 (20H10)},
  MRNUMBER = {2153905},
MRREVIEWER = {Natalia\ Kopteva},
       DOI = {10.1007/s00039-005-0511-1},
       URL = {https://doi.org/10.1007/s00039-005-0511-1},
}

@book {MR590044,
    AUTHOR = {Abikoff, William},
     TITLE = {The real analytic theory of {T}eichm\"uller space},
    SERIES = {Lecture Notes in Mathematics},
    VOLUME = {820},
 PUBLISHER = {Springer, Berlin},
      YEAR = {1980},
     PAGES = {vii+144},
      ISBN = {3-540-10237-X},
   MRCLASS = {32G15 (30F99 57N05)},
  MRNUMBER = {590044},
MRREVIEWER = {L.\ Keen},
}

@book {MR2742784,
    AUTHOR = {Buser, Peter},
     TITLE = {Geometry and spectra of compact {R}iemann surfaces},
    SERIES = {Modern Birkh\"auser Classics},
      NOTE = {Reprint of the 1992 edition},
 PUBLISHER = {Birkh\"auser Boston, Ltd., Boston, MA},
      YEAR = {2010},
     PAGES = {xvi+454},
      ISBN = {978-0-8176-4991-3},
   MRCLASS = {58J50 (30F10 32G15 58J53)},
  MRNUMBER = {2742784},
       DOI = {10.1007/978-0-8176-4992-0},
       URL = {https://doi.org/10.1007/978-0-8176-4992-0},
}

@book {MR3053012,
    AUTHOR = {Fathi, Albert and Laudenbach, Fran\c{c}ois and Po\'enaru,
              Valentin},
     TITLE = {Thurston's work on surfaces},
    SERIES = {Mathematical Notes},
    VOLUME = {48},
      NOTE = {Translated from the 1979 French original by Djun M. Kim and
              Dan Margalit},
 PUBLISHER = {Princeton University Press, Princeton, NJ},
      YEAR = {2012},
     PAGES = {xvi+254},
      ISBN = {978-0-691-14735-2},
   MRCLASS = {57-02 (30Fxx 32G15 57M99)},
  MRNUMBER = {3053012},
MRREVIEWER = {Javier\ Aramayona},
}

@article {MR4466646,
    AUTHOR = {Futer, David and Purcell, Jessica S. and Schleimer, Saul},
     TITLE = {Effective bilipschitz bounds on drilling and filling},
   JOURNAL = {Geom. Topol.},
  FJOURNAL = {Geometry \& Topology},
    VOLUME = {26},
      YEAR = {2022},
    NUMBER = {3},
     PAGES = {1077--1188},
      ISSN = {1465-3060,1364-0380},
   MRCLASS = {57K10 (30F40 57K32)},
  MRNUMBER = {4466646},
MRREVIEWER = {Ken-ichi\ Ohshika},
       DOI = {10.2140/gt.2022.26.1077},
       URL = {https://doi.org/10.2140/gt.2022.26.1077},
}

@article {MR4468992,
    AUTHOR = {Futer, David and Purcell, Jessica S. and Schleimer, Saul},
     TITLE = {Effective drilling and filling of tame hyperbolic 3-manifolds},
   JOURNAL = {Comment. Math. Helv.},
  FJOURNAL = {Commentarii Mathematici Helvetici. A Journal of the Swiss
              Mathematical Society},
    VOLUME = {97},
      YEAR = {2022},
    NUMBER = {3},
     PAGES = {457--512},
      ISSN = {0010-2571,1420-8946},
   MRCLASS = {57K32 (30F40)},
  MRNUMBER = {4468992},
MRREVIEWER = {Ken-ichi\ Ohshika},
       DOI = {10.4171/cmh/536},
       URL = {https://doi.org/10.4171/cmh/536},
}

@article {MR648524,
    AUTHOR = {Thurston, William P.},
     TITLE = {Three-dimensional manifolds, {K}leinian groups and hyperbolic
              geometry},
   JOURNAL = {Bull. Amer. Math. Soc. (N.S.)},
  FJOURNAL = {American Mathematical Society. Bulletin. New Series},
    VOLUME = {6},
      YEAR = {1982},
    NUMBER = {3},
     PAGES = {357--381},
      ISSN = {0273-0979,1088-9485},
   MRCLASS = {57N10 (20H15 30F40 57M35 57M40 57S17)},
  MRNUMBER = {648524},
MRREVIEWER = {Klaus\ Johannson},
       DOI = {10.1090/S0273-0979-1982-15003-0},
       URL = {https://doi.org/10.1090/S0273-0979-1982-15003-0},
}

@article {MR2680207,
    AUTHOR = {Agol, Ian},
     TITLE = {Bounds on exceptional {D}ehn filling {II}},
   JOURNAL = {Geom. Topol.},
  FJOURNAL = {Geometry \& Topology},
    VOLUME = {14},
      YEAR = {2010},
    NUMBER = {4},
     PAGES = {1921--1940},
      ISSN = {1465-3060,1364-0380},
   MRCLASS = {57M50 (30F40)},
  MRNUMBER = {2680207},
MRREVIEWER = {James\ W.\ Anderson},
       DOI = {10.2140/gt.2010.14.1921},
       URL = {https://doi.org/10.2140/gt.2010.14.1921},
}

@article {MR1218098,
    AUTHOR = {Bowditch, B. H.},
     TITLE = {Geometrical finiteness for hyperbolic groups},
   JOURNAL = {J. Funct. Anal.},
  FJOURNAL = {Journal of Functional Analysis},
    VOLUME = {113},
      YEAR = {1993},
    NUMBER = {2},
     PAGES = {245--317},
      ISSN = {0022-1236,1096-0783},
   MRCLASS = {57M50 (20H10 30F40 57S30)},
  MRNUMBER = {1218098},
MRREVIEWER = {Darryl\ McCullough},
       DOI = {10.1006/jfan.1993.1052},
       URL = {https://doi.org/10.1006/jfan.1993.1052},
}

@misc{lecuire2025propernessbendingmap,
      title={Properness of the bending map}, 
      author={Cyril Lecuire},
      year={2025},
      eprint={2510.07087},
      archivePrefix={arXiv},
      primaryClass={math.GT},
      url={https://arxiv.org/abs/2510.07087}, 
      note = {arXiv:2510.07087},
}

@incollection {MR2258744,
    AUTHOR = {Lecuire, Cyril},
     TITLE = {An extension of the {M}asur domain},
 BOOKTITLE = {Spaces of {K}leinian groups},
    SERIES = {London Math. Soc. Lecture Note Ser.},
    VOLUME = {329},
     PAGES = {49--73},
 PUBLISHER = {Cambridge Univ. Press, Cambridge},
      YEAR = {2006},
      ISBN = {978-0-521-61797-0; 0-521-61797-9},
   MRCLASS = {57N10 (30F40 57M50)},
  MRNUMBER = {2258744},
MRREVIEWER = {Ken-ichi\ Ohshika},
}

@article {MR1678469,
    AUTHOR = {Bonahon, Francis},
     TITLE = {Variations of the boundary geometry of {$3$}-dimensional
              hyperbolic convex cores},
   JOURNAL = {J. Differential Geom.},
  FJOURNAL = {Journal of Differential Geometry},
    VOLUME = {50},
      YEAR = {1998},
    NUMBER = {1},
     PAGES = {1--24},
      ISSN = {0022-040X,1945-743X},
   MRCLASS = {57N10 (30F45 53C22 57M50)},
  MRNUMBER = {1678469},
MRREVIEWER = {Dubravko\ Ivan\v si\'c},
       URL = {http://projecteuclid.org/euclid.jdg/1214510044},
}

@article {MR251714,
    AUTHOR = {Lacher, R. C.},
     TITLE = {Cell-like mappings. {I}},
   JOURNAL = {Pacific J. Math.},
  FJOURNAL = {Pacific Journal of Mathematics},
    VOLUME = {30},
      YEAR = {1969},
     PAGES = {717--731},
      ISSN = {0030-8730,1945-5844},
   MRCLASS = {55.25 (54.00)},
  MRNUMBER = {251714},
MRREVIEWER = {D.\ R.\ McMillan, Jr.},
       URL = {http://projecteuclid.org/euclid.pjm/1102978255},
}

@article {MR2188131,
    AUTHOR = {Calegari, Danny and Gabai, David},
     TITLE = {Shrinkwrapping and the taming of hyperbolic 3-manifolds},
   JOURNAL = {J. Amer. Math. Soc.},
  FJOURNAL = {Journal of the American Mathematical Society},
    VOLUME = {19},
      YEAR = {2006},
    NUMBER = {2},
     PAGES = {385--446},
      ISSN = {0894-0347,1088-6834},
   MRCLASS = {57M50 (30F40 57N10)},
  MRNUMBER = {2188131},
MRREVIEWER = {Bruno\ P.\ Zimmermann},
       DOI = {10.1090/S0894-0347-05-00513-8},
       URL = {https://doi.org/10.1090/S0894-0347-05-00513-8},
}

@misc{agol2004tamenesshyperbolic3manifolds,
      title={Tameness of hyperbolic 3-manifolds}, 
      author={Ian Agol},
      year={2004},
      eprint={math/0405568},
      archivePrefix={arXiv},
      primaryClass={math.GT},
      url={https://arxiv.org/abs/math/0405568}, 
}

@article {MR326737,
    AUTHOR = {Scott, G. P.},
     TITLE = {Compact submanifolds of {$3$}-manifolds},
   JOURNAL = {J. London Math. Soc. (2)},
  FJOURNAL = {Journal of the London Mathematical Society. Second Series},
    VOLUME = {7},
      YEAR = {1973},
     PAGES = {246--250},
      ISSN = {0024-6107,1469-7750},
   MRCLASS = {57A10},
  MRNUMBER = {326737},
MRREVIEWER = {C.\ D.\ Feustel},
       DOI = {10.1112/jlms/s2-7.2.246},
       URL = {https://doi.org/10.1112/jlms/s2-7.2.246},
}

@article {MR224099,
    AUTHOR = {Waldhausen, Friedhelm},
     TITLE = {On irreducible {$3$}-manifolds which are sufficiently large},
   JOURNAL = {Ann. of Math. (2)},
  FJOURNAL = {Annals of Mathematics. Second Series},
    VOLUME = {87},
      YEAR = {1968},
     PAGES = {56--88},
      ISSN = {0003-486X},
   MRCLASS = {57.05},
  MRNUMBER = {224099},
MRREVIEWER = {W.\ Haken},
       DOI = {10.2307/1970594},
       URL = {https://doi.org/10.2307/1970594},
}

@book {MR551744,
    AUTHOR = {Johannson, Klaus},
     TITLE = {Homotopy equivalences of {$3$}-manifolds with boundaries},
    SERIES = {Lecture Notes in Mathematics},
    VOLUME = {761},
 PUBLISHER = {Springer, Berlin},
      YEAR = {1979},
     PAGES = {ii+303},
      ISBN = {3-540-09714-7},
   MRCLASS = {57N10},
  MRNUMBER = {551744},
MRREVIEWER = {John\ Hempel},
}

@article {MR825931,
    AUTHOR = {McCullough, D. and Miller, A. and Swarup, G. A.},
     TITLE = {Uniqueness of cores of noncompact {$3$}-manifolds},
   JOURNAL = {J. London Math. Soc. (2)},
  FJOURNAL = {Journal of the London Mathematical Society. Second Series},
    VOLUME = {32},
      YEAR = {1985},
    NUMBER = {3},
     PAGES = {548--556},
      ISSN = {0024-6107,1469-7750},
   MRCLASS = {57N10 (57N35)},
  MRNUMBER = {825931},
MRREVIEWER = {J\'ozef\ H.\ Przytycki},
       DOI = {10.1112/jlms/s2-32.3.548},
       URL = {https://doi.org/10.1112/jlms/s2-32.3.548},
}

@article {MR840832,
    AUTHOR = {McCullough, Darryl and Miller, Andy},
     TITLE = {Homeomorphisms of {$3$}-manifolds with compressible boundary},
   JOURNAL = {Mem. Amer. Math. Soc.},
  FJOURNAL = {Memoirs of the American Mathematical Society},
    VOLUME = {61},
      YEAR = {1986},
    NUMBER = {344},
     PAGES = {xii+100},
      ISSN = {0065-9266,1947-6221},
   MRCLASS = {57N10 (55P10 55S37 57M99 57R50)},
  MRNUMBER = {840832},
MRREVIEWER = {Don\v co\ Dimovski},
       DOI = {10.1090/memo/0344},
       URL = {https://doi.org/10.1090/memo/0344},
}

@article {MR442293,
    AUTHOR = {Abikoff, William},
     TITLE = {Degenerating families of {R}iemann surfaces},
   JOURNAL = {Ann. of Math. (2)},
  FJOURNAL = {Annals of Mathematics. Second Series},
    VOLUME = {105},
      YEAR = {1977},
    NUMBER = {1},
     PAGES = {29--44},
      ISSN = {0003-486X},
   MRCLASS = {32G15 (14H15)},
  MRNUMBER = {442293},
MRREVIEWER = {C.\ Earle},
       DOI = {10.2307/1971024},
       URL = {https://doi.org/10.2307/1971024},
}

@article {MR349989,
    AUTHOR = {Hejhal, Dennis A.},
     TITLE = {Universal covering maps for variable regions},
   JOURNAL = {Math. Z.},
  FJOURNAL = {Mathematische Zeitschrift},
    VOLUME = {137},
      YEAR = {1974},
     PAGES = {7--20},
      ISSN = {0025-5874,1432-1823},
   MRCLASS = {30A46},
  MRNUMBER = {349989},
MRREVIEWER = {C.\ Earle},
       DOI = {10.1007/BF01213931},
       URL = {https://doi.org/10.1007/BF01213931},
}

@book{comar1996hyperbolic,
  title={Hyperbolic Dehn surgery and convergence of Kleinian groups},
  author={Comar, Timothy D},
  year={1996},
  publisher={University of Michigan}
}

@article {MR173265,
    AUTHOR = {Lojasiewicz, S.},
     TITLE = {Triangulation of semi-analytic sets},
   JOURNAL = {Ann. Scuola Norm. Sup. Pisa Cl. Sci. (3)},
  FJOURNAL = {Annali della Scuola Normale Superiore di Pisa. Classe di
              Scienze. Serie III},
    VOLUME = {18},
      YEAR = {1964},
     PAGES = {449--474},
      ISSN = {0391-173X},
   MRCLASS = {32.25},
  MRNUMBER = {173265},
MRREVIEWER = {S.\ S.\ Cairns},
}

@book {MR1219310,
    AUTHOR = {Benedetti, Riccardo and Petronio, Carlo},
     TITLE = {Lectures on hyperbolic geometry},
    SERIES = {Universitext},
 PUBLISHER = {Springer-Verlag, Berlin},
      YEAR = {1992},
     PAGES = {xiv+330},
      ISBN = {3-540-55534-X},
   MRCLASS = {57M50 (30F40 30F60 51M10 57N10)},
  MRNUMBER = {1219310},
MRREVIEWER = {Colin\ C.\ Adams},
       DOI = {10.1007/978-3-642-58158-8},
       URL = {https://doi.org/10.1007/978-3-642-58158-8},
}

@article {MR857678,
    AUTHOR = {Douady, Adrien and Earle, Clifford J.},
     TITLE = {Conformally natural extension of homeomorphisms of the circle},
   JOURNAL = {Acta Math.},
  FJOURNAL = {Acta Mathematica},
    VOLUME = {157},
      YEAR = {1986},
    NUMBER = {1-2},
     PAGES = {23--48},
      ISSN = {0001-5962,1871-2509},
   MRCLASS = {30C60 (32G15 58F08)},
  MRNUMBER = {857678},
MRREVIEWER = {William\ Abikoff},
       DOI = {10.1007/BF02392590},
       URL = {https://doi.org/10.1007/BF02392590},
}

@article {MR781586,
    AUTHOR = {Tukia, Pekka},
     TITLE = {Quasiconformal extension of quasisymmetric mappings compatible
              with a {M}\"obius group},
   JOURNAL = {Acta Math.},
  FJOURNAL = {Acta Mathematica},
    VOLUME = {154},
      YEAR = {1985},
    NUMBER = {3-4},
     PAGES = {153--193},
      ISSN = {0001-5962,1871-2509},
   MRCLASS = {30C60 (30F35)},
  MRNUMBER = {781586},
MRREVIEWER = {Gaven\ J.\ Martin},
       DOI = {10.1007/BF02392471},
       URL = {https://doi.org/10.1007/BF02392471},
}

@article {MR1459103,
    AUTHOR = {Alexander, Stephanie B. and Bishop, Richard L.},
     TITLE = {The {F}ary-{M}ilnor theorem in {H}adamard manifolds},
   JOURNAL = {Proc. Amer. Math. Soc.},
  FJOURNAL = {Proceedings of the American Mathematical Society},
    VOLUME = {126},
      YEAR = {1998},
    NUMBER = {11},
     PAGES = {3427--3436},
      ISSN = {0002-9939,1088-6826},
   MRCLASS = {53C22 (53C40 57M25)},
  MRNUMBER = {1459103},
MRREVIEWER = {Wolfgang\ K\"uhnel},
       DOI = {10.1090/S0002-9939-98-04423-2},
       URL = {https://doi.org/10.1090/S0002-9939-98-04423-2},
}

@article{dular2025bending,
  title={Bending parameterization of one-sided degenerated Kleinian surface groups},
  author={Dular, Bruno},
  journal={arXiv preprint arXiv:2504.19891},
  year={2025}
}

@incollection {MR758464,
    AUTHOR = {Morgan, John W.},
     TITLE = {On {T}hurston's uniformization theorem for three-dimensional
              manifolds},
 BOOKTITLE = {The {S}mith conjecture ({N}ew {Y}ork, 1979)},
    SERIES = {Pure Appl. Math.},
    VOLUME = {112},
     PAGES = {37--125},
 PUBLISHER = {Academic Press, Orlando, FL},
      YEAR = {1984},
      ISBN = {0-12-506980-4},
   MRCLASS = {57N10 (30F10)},
  MRNUMBER = {758464},
       DOI = {10.1016/S0079-8169(08)61637-2},
       URL = {https://doi.org/10.1016/S0079-8169(08)61637-2},
}

@book {MR2553578,
    AUTHOR = {Kapovich, Michael},
     TITLE = {Hyperbolic manifolds and discrete groups},
    SERIES = {Modern Birkh\"auser Classics},
      NOTE = {Reprint of the 2001 edition},
 PUBLISHER = {Birkh\"auser Boston, Ltd., Boston, MA},
      YEAR = {2009},
     PAGES = {xxviii+467},
      ISBN = {978-0-8176-4912-8},
   MRCLASS = {57M50 (20F65 20H10 30F40 30F45 30F60 32G15)},
  MRNUMBER = {2553578},
       DOI = {10.1007/978-0-8176-4913-5},
       URL = {https://doi.org/10.1007/978-0-8176-4913-5},
}

@article {MR1622600,
    AUTHOR = {Hodgson, Craig D. and Kerckhoff, Steven P.},
     TITLE = {Rigidity of hyperbolic cone-manifolds and hyperbolic {D}ehn
              surgery},
   JOURNAL = {J. Differential Geom.},
  FJOURNAL = {Journal of Differential Geometry},
    VOLUME = {48},
      YEAR = {1998},
    NUMBER = {1},
     PAGES = {1--59},
      ISSN = {0022-040X,1945-743X},
   MRCLASS = {57M50 (53C21)},
  MRNUMBER = {1622600},
MRREVIEWER = {Athanase\ Papadopoulos},
       URL = {http://projecteuclid.org/euclid.jdg/1214460606},
}
